\documentclass[reqno]{amsart}
\usepackage{amsfonts,amsmath,amsthm}
\usepackage{latexsym}
\usepackage{array}
\usepackage{amssymb}
\usepackage{graphicx}
\usepackage{enumerate}
\usepackage{verbatim}
\usepackage{subfigure}
\usepackage{float}

\usepackage[english]{babel}

\usepackage{color}
\definecolor{marin}{rgb}   {0.,   0.3,   0.7} 
\definecolor{rouge}{rgb}   {0.8,   0.,   0.} 
\definecolor{sepia}{rgb}   {0.8,   0.5,   0.} 
\usepackage[colorlinks,citecolor=marin,linkcolor=rouge,
            bookmarksopen,
            bookmarksnumbered
           ]{hyperref}

\theoremstyle{plain} 
\newtheorem{theorem}{Theorem}[section]
\newtheorem{lemma}[theorem]{Lemma}

 \theoremstyle{remark}
\newtheorem{remark}[theorem]{Remark}
\newtheorem{example}[theorem]{Example}

\newcommand {\aplt} {\ {\raise-.5ex\hbox{$\buildrel<\over\sim$}}\ } 
\newcommand {\gplt} {\ {\raise-.5ex\hbox{$\buildrel>\over\sim$}}\ }

\newcommand{\T}{\mathbb{T}}\newcommand{\R}{\mathbb{R}}\newcommand{\C}{\mathbb{C}}
\newcommand{\Z}{\mathbb{Z}}\newcommand{\N}{\mathbb{N}}
\newcommand{\dd}{\mathrm{d}}
\newcommand{\nabo}{| \nabla |}
\newcommand{\sino}{ \sin(\tau| \nabla |)}
\newcommand{\coso} {\cos(\tau| \nabla |)}
\newcommand{\nabs}{\langle \nabla \rangle}

\numberwithin{equation}{section}

\begin{document}

\author{Sebastian Herr}
\address{Fakult\"{a}t f\"{u}r Mathematik, Universit\"{a}t Bielefeld,
  Postfach 10 01 31, 33501 Bielefeld, Germany}
\email{herr@math.uni-bielefeld.de}

\author{Katharina Schratz}
\address{Fakult\"{a}t f\"{u}r Mathematik, Karlsruhe Institute of Technology,
Englerstr. 2, 76131, Karlsruhe}
\email{katharina.schratz@kit.edu}

\begin{abstract}
The main challenge in the analysis of numerical schemes for the
Zakharov system originates from the presence of derivatives in the
nonlinearity. In this paper a new trigonometric time-integration
scheme for the Zakharov system is constructed and convergence is
proved.  The time-step restriction is independent from a spatial
discretization. Numerical experiments confirm the findings.
\end{abstract}


%\thanks{}

\title{Trigonometric time integrators for the Zakharov system}
\maketitle

\section{Introduction}\label{sect:intro}
We consider the Zakharov system
\begin{equation}\label{eq:ZakO}
\begin{aligned}
& i \partial_t E + \Delta E = u E,\\
& \partial_{tt} u - \Delta u = \Delta \vert E\vert^2,
\end{aligned}
\end{equation}
with initial conditions
\begin{equation}\label{eq:icO}
E(0) = E_0,\; u(0) = u_0,\; \partial_t u(0) = u_1,
\end{equation}
for given initial data $E_0,u_0,u_1$ in appropriate Sobolev spaces. This system is a scalar model for Langmuir oscilations in a plasma, see \cite{Su99,Zak72}. Here, $E:\R^{1+d}\to \C$ denotes the  (scalar) electric field envelope and $u:\R^{1+d}\to \R$ the ion density fluctuation in spatial dimension $d\in \N$.
For practical implementation reasons we impose periodic boundary
conditions, hence both $E$ and $u$ are considered to be spatially periodic.

The Zakharov system has a Hamiltonian structure and conserved quantities. More precisely, for strong solutions we have
\begin{equation}\label{eq:l2con}
\frac{\dd}{\dd t} \int_{\T^d} |E(t,x)|^2 dx=0
\end{equation}
and, if $u_1$ has mean zero,
\begin{equation}\label{eq:ham}
 \frac{\dd}{\dd t} \int_{\T^d} |\nabla E(t,x)|^2 + u(t,x) |E(t,x)|^2 + \frac{1}{2}||\nabla|^{-1}\partial_t u(t,x)|^2+\frac{1}{2}|u(t,x)|^2dx=0,
\end{equation}
where $\nabo = \sqrt{-\Delta}$ and $\T^d=\R^d/ (2\pi \Z)^d$. The latter is called conservation of energy.

Several time integrators for solving the Zakharov system numerically have been proposed. Due to the outstanding performance of splitting methods for nonlinear Schr\"odinger equations, see the recent papers \cite{Faou12,Gau11,Lubich08} and references therein, splitting methods for the generalized Zakharov system were constructed in \cite{Bao03,Bao05,JiMa04}. In \cite{PayND83} finite differences for the time discretization and a pseudo spectral method for the space discretization were used to simulate the Zakharov system numerically. Numerically, the above schemes have been tested extensively. However, due to the difficult structure of the system, as explained  below in more detail, a convergence analysis is missing.

For the one dimensional Zakharov equations fully-implicit and
semi-explicit Crank-Nicolson type approximations based on finite
difference in time and space  were derived in \cite{Glass92b,Glass92}
and \cite{ChangGJ95,ChangJ94}, respectively. Numerical experiments
\cite{ChangJ94} indicate that the semi-explicit method (which is
explicit in $n$ and implicit in $E$) is preferable over the fully
implicit method (which is both implicit in $n$ and in $E$) due to the
high computational costs of the latter. However, its convergence only
holds under the constraint $\Delta t = \Delta x$, where $\Delta t$ and $\Delta x$ denote the time and space discretization parameters. Furthermore, due to the use of the Sobolev embedding theorem the convergence results only hold in one dimension.

The main challenge in the construction and analysis of any numerical scheme for the Zakharov system \eqref{eq:ZakO} originates from the presence of derivatives in the nonlinearity: Mild solutions are given by
\begin{equation}\label{mildZo}
\begin{aligned}
E(t) =& \mathrm{e}^{i t \Delta} E(0) -i \int_0^t \mathrm{e}^{i (t-\xi)\Delta} u(\xi) E(\xi)\dd \xi,\\
 u(t) =& \cos (t \nabo)u(0) + \frac{\sin (t\nabo)}{\nabo} u'(0)+\int_0^t \sin((t-\xi)\nabo) \nabo \vert E(\xi)\vert^2 \dd \xi.
\end{aligned}
\end{equation}
However, it is not obvious how to bound the quadratic term $\nabo \vert E \vert^2$, since
``naively'' estimating the solutions yields
\begin{equation*}
\begin{aligned}
  & \Vert E(t)\Vert_s \leq \Vert E(0)\Vert_s + c\int_0^t \Vert u (\xi) \Vert_s \Vert E(\xi)\Vert_s \dd \xi,\quad &s > d/2,\\
& \Vert u(t)\Vert_l \leq \Vert u(0)\Vert_l + \Vert u'(0)\Vert_{l-1}+ c\int_0^t\Vert E(\xi)\Vert_{l+1}^2\dd \xi,\quad & l+1>d/2,
\end{aligned}
\end{equation*}
which amounts to a \emph{loss of derivatives}, see Section \ref{sect:err} for a definition of $\|\cdot\|_s$.

In order to avoid this, we follow the strategy presented in \cite{OzaTsu92}: We reformulate the Zakharov system as a system in $(E,\partial_t E, u, \partial_t u)$. This allows us to construct trigonometric time-integration schemes for the Zakharov system~\eqref{eq:ZakO} without imposing any spatial-dependent time-step condition or too restrictive regularity assumptions on the initial data (such as analyticity). In particular, their convergence also holds in the limit $\Delta x \to 0$.

For recent developments in trigonometric and exponential integration schemes for wave-type equations we refer to \cite{Gau15,HLW,HochL99,HochOst10} and the references therein. For local-wellposedness of the Zakharov system in Sobolev spaces of low regularity on $\T^d$ we refer to \cite{Bourg94,Tak99,Kish13}. Concerning the well-posedness theory on $\R^d$ we refer to \cite{OzaTsu92,bc-96,gtv-97,bhht-09,bh-11} and references therein.

\section{Trigonometric integrators for the Zakharov system}%\label{sect:tri}
To avoid the \emph{loss of derivatives} we use the method devised in \cite{OzaTsu92}: We reformulate the Zakharov system \eqref{eq:ZakO} as
\begin{equation}\label{eq:Zak}
\begin{aligned}
& i \partial_t F + \Delta F = u F + \partial_t u \left(E(0)+\int_0^t F(\xi) \dd \xi\right)\\
& \partial_{tt} u - \Delta u = \Delta \vert E\vert^2,\\
&( -\Delta +1) E = i F- (u-1) \left(E(0) + \int_0^t F(\xi)\dd \xi\right),
\end{aligned}
\end{equation}
where $ F = \partial_t E$ (cf. \cite{OzaTsu92}), with initial
conditions
\begin{equation}\label{eq:ic}
F(0) = i \big(\Delta E(0) - u(0)E(0)\big), \; u(0)=u_0,
\; \partial_tu(0)=u_1, \; E(0)=E_0.
\end{equation}
Let
\begin{equation}\label{IF}
 \mathcal{I}_F(t):=  E_0+\int_0^{t} F(\lambda)\dd \lambda.
\end{equation}
Then the mild solutions of \eqref{eq:Zak} at time $t_{n+1}= t_n + \tau$ with $t_0 =0$ read
\begin{equation}\label{eq:solZ}
\begin{split}
F(t_n+\tau)=&\mathrm{e}^{i \tau \Delta} F(t_n) -i\int_0^\tau \mathrm{e}^{i (\tau-\xi)\Delta} \Big (
(uF+u' \mathcal{I}_F)(t_n+\xi) \Big)\dd \xi\\
 u(t_n+\tau) =& \cos (\tau \nabo)u(t_n) + \nabo^{-1} \sin (\tau \nabo) u'(t_n)\\
&+\int_0^\tau \nabo^{-1} \sin((\tau-\xi)\nabo) \Delta \vert E(t_n+\xi)\vert^2 \dd \xi,\\
  u'(t_n+\tau) =& - \nabo\sin (\tau \nabo)u(t_n) +\cos (\tau \nabo) u'(t_n)\\
&+  \int_0^\tau  \cos((\tau-\xi)\nabo) \Delta \vert E(t_n+\xi)\vert^2 \dd \xi,\\
 E(t_n+\tau) =& (1 -\Delta)^{-1}\big(i F(t_n+\tau)-(u(t_n+\tau)-1)\mathcal{I}_F(t_n+\tau)\big).
\end{split}
\end{equation}
In Section \ref{sect:tri} we develop a first-order trigonometric
integration scheme based on the reformulation \eqref{eq:solZ} and
rigorously carry out its convergence analysis. Furthermore, in Section
\ref{sec:2app} we indicate a generalization to a second-order
trigonometric integration scheme.

\section{A first-order scheme}\label{sect:tri}
In order to construct a robust first-order scheme we approximate the
exact solutions $(u,u',F,E)(t_n+\xi)$ appearing in the integrals in \eqref{eq:solZ}
via Taylor series expansion up to the first-order remainder term. This allows us to integrate $\mathrm{e}^{i \xi \Delta},\mathrm{cos}( \xi \Delta)$ and $\mathrm{sin}(\xi\Delta)$ exactly. Furthermore, we use the following approximation for the integrals over $F$: Note that for $0\leq \xi \leq \tau$
\begin{equation}\label{intF}
\begin{aligned}
\int_0^{t_n+\xi} F(\lambda) \dd \lambda&=
\sum_{k=0}^{n-1}\int_{t_k}^{t_{k+1}} F(\lambda)\dd\lambda +
\int_{t_n}^{t_n+\xi} F(\lambda)\dd\lambda\\
& =
\sum_{k=0}^{n-1}\int_0^\tau F(t_k+\lambda)\dd\lambda + \int_0^\xi
F(t_n+\lambda)\dd\lambda
\\
&=\tau \sum_{k=0}^{n} F(t_k) +\mathcal{F}_{\tau,\xi,n},
\end{aligned}
\end{equation}
where
\begin{align*}
\mathcal{F}_{\tau,\xi,n}:=(\xi-\tau)F(t_n)+\sum_{k=0}^{n-1} \int_0^\tau
\int_0^\lambda F'(t_k+r)\dd r\dd \lambda +\int_0^\xi
\int_0^\lambda F'(t_n+r)\dd r\dd \lambda.
\end{align*}
We observe that
\begin{equation}\label{intF2}
\|\mathcal{F}_{\tau,\xi,n}\|_s\leq \tau \|F(t_n)\|_s+ \tau t_n \sup_{t\in [0,t_{n+1}]}\|F'(t)\|_s,
\end{equation}
and it this sense we have, for $0\leq \xi\leq \tau$,
\[
\int_0^{t_n+\xi} F(\lambda) \dd \lambda \approx  \tau \sum_{k=0}^{n} F(t_k).
\]
Recall the initial conditions \eqref{eq:ic}. By setting
\begin{equation}\label{eq:init}
E^0 = E_0,\quad u^0=u_0,\quad u'^0 = u_1, \quad F^0 = i (\Delta E^0-u^0E^0),\quad S_F^0 = E_0+\tau F^0,
\end{equation}
we obtain, for $n\geq 0$, the first-oder trigonometric time-integration scheme
\begin{equation}\label{eq:numZ}
\begin{aligned}
F^{n+1} &= \mathrm{e}^{i \tau \Delta} F^n +i \tau \frac{1-\mathrm{e}^{i\tau\Delta}}{i\tau \Delta}
 \left (
u^n F^n+u'^n  S_F^n \right),\\
 u^{n+1} & = \cos (\tau \nabo)u^n + \nabo^{-1} \sin (\tau \nabo) u'^n+\tau\nabo^{-1}\frac{1- \cos(\tau \nabo)}{\tau\nabo} \Delta \vert E^n\vert^2,\\
    u'^{n+1} & = - \nabo\sin (\tau \nabo)u^n +\cos (\tau \nabo) u'^n+\tau \frac{\sin(\tau\nabo)}{\tau\nabo} \Delta \vert E^n\vert^2,\\
    S_F^{n+1} & = S_F^n + \tau F^{n+1},\\
E^{n+1} & =   (-\Delta+1)^{-1}\left(  i F^{n+1} -(u^{n+1}-1) S_F^{n+1}\right),
\end{aligned}
\end{equation}
\begin{remark}
Note that for given $(E^n,F^n,u^n,u'^n, S_F^n)$ we can compute the next iteration without saving $(E^k,F^k,u^k,u'^k, S_F^k)$ for any $k < n$.
\end{remark}

\begin{remark}
For initial data of sufficiently high Sobolev regularity we will prove that the scheme \eqref{eq:numZ} is of first-order. Note that one can also use higher order quadrature formulas to generate higher order schemes, given additional smoothness of the initial data. We give a generalization to a second-order scheme in Section \ref{sec:2app}.
\end{remark}
\subsection{Error analysis}\label{sect:err}
In this section we carry out the error analysis of the trigonometric
time-integration scheme \eqref{eq:numZ}. In the following we set for  $f(x) = \sum_{k\in \mathbb{Z}^d} \hat{f}(k) \mathrm{e}^{i k \cdot x}$ and $s \in \mathbb{R}$
\begin{equation*}
\begin{aligned}
& \nabo^s f(x) := \sum_{k \in \mathbb{Z}^d} \vert k\vert^s \hat{f}(k) \mathrm{e}^{i k \cdot x } ,\quad
\langle \nabla \rangle^s f(x) := \nabo^s f(x) + \hat{f}(0)
\end{aligned}
\end{equation*}
and define
$$
\Vert f \Vert_s := \Vert \langle \nabla \rangle^s f \Vert_{L^2(\mathbb{T}^d)}.
$$
For $s>d/2$, we will exploit the fact that $H^s(\T^d)$ is an algebra, with the standard product estimate
\[\Vert f g \Vert_s \leq c \Vert f \Vert_s \Vert g \Vert_s,\]
where $c$ only depends on $d$ and $s$. Furthermore, we denote by $\mathcal{L}(X)$ the space of bounded linear operators $T:X\to X$, and sometimes we write $\|T \|_{s}$ instead of $\|T \|_{\mathcal{L}(H^s(\T^d))}$ for the sake of brevity.

In view of the structure of the Zakharov system
\[
\|(E(t),u(t),u'(t))\|_{[s]}:=\Vert E(t)\Vert_{s+2} + \Vert u(t)\Vert_{s+1} + \Vert u'(t)\Vert_s
\]
is the natural norm for our error analysis, the auxiliary function $F$ will be measured in $\|\cdot\|_s$ then.
\begin{theorem}\label{thm:s}
 Fix $s>d/2$ and $0<\gamma \leq 1$. For any $T\in (0,\infty)$, suppose that
\begin{equation*}
E \in \mathcal{C}([0,T];H^{s+2+2\gamma}(\mathbb{T}^d)),\quad u \in \mathcal{C}([0,T]; H^{s+1+2\gamma}(\mathbb{T}^d)) \cap \mathcal{C}^1([0,T]; H^{s+2\gamma}(\mathbb{T}^d))
\end{equation*}
is a mild solution of \eqref{eq:ZakO} with
 \begin{equation}\label{regTaylor}
 m_{s+2\gamma}(T) := \sup_{t\in [0,T]} \|(E(t),u(t),u'(t))\|_{[s+2\gamma]}<\infty.
 \end{equation}
Then, there exists $\tau_0>0$ such that for all $0\leq \tau\leq \tau_0$ and $t_n = n\tau \leq T$ the trigonometric time-integration scheme \eqref{eq:numZ} is convergent of order $\gamma$, i.e.,
$$ 
\|(E(t_n)-E^n, u(t_n)-u^n,u'(t_n)-u'^n)\|_{[s]} \leq e^{c_1}c_2 \tau^\gamma  ,
$$
where $c_1$ and $c_2$ depend only on $m_{s}(T)$ and $m_{s+2\gamma}(T)$, respectively, as well as on $T$,  $d$ and $s$.
\end{theorem}
\begin{remark} Theorem \ref{thm:s} implies first-order convergence in the case $\gamma=1$.
\end{remark}

\begin{remark}\label{rem:lwp}
Note that the Zakharov system \eqref{eq:ZakO} is locally well-posed in the space
$$
H^s(\T^d) \times H^\ell(\T^d)\times H^{\ell-1}(\T^d) \ni (E,u,u'),$$
provided that
\begin{equation}\label{lwzd}
\begin{aligned}
& 0\leq s-\ell \leq 1, \quad 1/2 \leq \ell+1/2 \leq 2s,\quad \text{for } d = 1,\\
& 0\leq s-\ell \leq 1, \quad 1 \leq \ell+1 \leq 2s,\quad\quad\quad \text{for } d = 2,\\
& 0\leq s-\ell \leq 1, \quad d-1<  \ell+d/2 \leq 2s,~ \text{for } d \geq 3,
\end{aligned}
\end{equation}
see \cite{Kish13,Tak99}. Hence, for $0 \leq \gamma \leq 1$ and
\begin{equation*}\label{regAN}
\Vert E(0)\Vert_{s+2+2\gamma}+ \Vert u(0)\Vert_{s+1+2\gamma}+ \Vert u'(0)\Vert_{s+2\gamma} \leq M_\gamma
\end{equation*}
there exists a $T_0=T_0(M_\gamma)>0$ such that
\begin{equation*}\label{correg}
E \in \mathcal{C}([0,T_0];H^{s+2+2\gamma}(\mathbb{T}^d)),\quad u \in \mathcal{C}([0,T_0]; H^{s+1+2\gamma}(\mathbb{T}^d)) \cap \mathcal{C}^1([0,T_0]; H^{s+2\gamma}(\mathbb{T}^d)),
\end{equation*}
which implies that \eqref{regTaylor} holds at least for $T = T_0$.
\end{remark}

\begin{proof}[Proof of Theorem \ref{thm:s}]
Let $s > d/2$. Due to the fact that $H^s(\T^d)$ is an algebra it is easy to see that the mild solution $(E,u)$ of \eqref{eq:ZakO} satisfies $\partial_t E \in \mathcal{C}([0,T];H^{s+2\gamma}(\mathbb{T}^d))$ and that $(F,E,u)$ solves \eqref{eq:solZ} for $F=\partial_t E$, see above.
In the following, $c$ denotes a generic constant which depends on $d$ and $s$ only. We will prove the claim for $n+1$ instead of $n$. Subtracting the numerical solutions \eqref{eq:numZ} from the exact solutions \eqref{eq:solZ} yields
\begin{equation}\label{ErrorrecTaylor-f}
\begin{split}
 F(t_{n+1})-F^{n+1}
  ={}& \mathrm{e}^{i\tau\Delta} (F(t_n)-F^n) \\&+ i \tau\frac{1-\mathrm{e}^{i\tau\Delta}}{i\tau\Delta} \Big(u(t_n)(F(t_n)-F^n) + (u(t_n)-u^n)F^n\\& + (u'(t_n)-u'^n)E(0)+ u'(t_n)\big(\tau\sum_{k=0}^n(F(t_k)-F^k)\big) \\&+(u'(t_n)-u'^n) (\tau\sum_{k=0}^nF^k)\Big) + L_F^n,
\end{split}
\end{equation}
and
\begin{equation}\label{ErrorrecTaylor-u}
\begin{split}
\nabs (u(t_{n+1})-u^{n+1})  ={}& \coso \nabs (u(t_n)-u^n) \\
&+ \sino \frac{\nabs}{\nabo} (u'(t_n)-u'^n) \\&+ \tau\frac{1- \cos(\tau \nabo)}{\tau\nabo}\frac{\nabs}{\nabo}  \Delta \left( \vert E(t_n)\vert^2  - \vert E^n\vert^2\right) \\&+ \nabs L_u^n,
\end{split}
\end{equation}
as well as
\begin{equation}\label{ErrorrecTaylor-uprime}
\begin{split}
u'(t_{n+1})-u'^{n+1} ={}&- \sino \frac{\nabo}{\nabs}  \nabs (u(t_n)-u^n) \\
&+ \coso (u'(t_n)-u'^n) \\&+\tau \frac{\sin(\tau\nabo)}{\tau\nabo} \Delta \left( \vert E(t_n)\vert^2  - \vert E^n\vert^2\right) + L_{u'}^n,
\end{split}
\end{equation}
and
\begin{equation}\label{ErrorrecTaylor-e}
\begin{split}
E(t_{n+1}) - E^{n+1}  ={}& (-\Delta+1)^{-1}\Big( i (F(t_{n+1}) - F^{n+1}) \\&- (u(t_{n+1})-u^{n+1})\big(E(0)+\tau \sum_{k=0}^n F^{k+1})\big)\\
& + (1-u(t_{n+1})) \big(\tau \sum_{k=0}^n (F(t_{k+1})-F^{k+1})\big) + \Delta L_E^n\Big).
\end{split}
\end{equation}
The local errors at time $t_n$ satisfy
\begin{equation}\label{eq:locErr}
\begin{split}
 \Vert L_F^n \Vert_s ={}& \Big\Vert \int_0^\tau \mathrm{e}^{i(\tau-\xi)\Delta} \Big( u(t_n+\xi)F(t_n+\xi) - u(t_n) F(t_n) \\&+ u'(t_n+\xi) \mathcal{I}_F(t_n+\xi)-  u'(t_n) \big(E(0) + \tau \sum_{k=0}^n F(t_k)\big) \Big) \dd\xi \Big\Vert_s,\\
 \Vert \nabs L_u^n \Vert_{s} ={}&\Big\Vert \frac{\nabs}{\nabo}  \int_0^\tau \sin((\tau-\xi)\nabo) \Delta  \left(\vert E(t_n+\xi)\vert^2 -\vert E(t_n)\vert^2\right) \dd \xi \Big\Vert_{s},\\
 \Vert L_{u'}^n \Vert_s={}&  \Big\Vert \int_0^\tau \cos((\tau-\xi)\nabo) \Delta  \left(\vert E(t_n+\xi)\vert^2 -\vert E(t_n)\vert^2\right) \dd \xi \Big\Vert_{s},\\
 \Vert \Delta L_E^n\Vert_s ={}& \Big\Vert\big (1-u(t_{n+1})\big)\Big( \int_0^{t_n+\tau} F(\lambda)\dd \lambda - \tau \sum_{k=0}^n F(t_{k+1})\Big)\Big\Vert_s.
\end{split}
\end{equation}
By Lemma \ref{lem:locerror} below we have
\begin{equation}\label{eq:loce}
\max_{0\leq k \leq n} \big\{ \Vert L_F^k\Vert_s + \Vert \nabs L_u^k\Vert_s + \Vert L_{u'}^k\Vert_s + \tau \Vert \Delta L_E^k\Vert_s\big \} \leq c \tau^{1+\gamma} t_n(1+m_{s+2\gamma}(T))^4.
\end{equation}
Hence, the local errors \eqref{eq:locErr} are of order $\tau^{1+\gamma}$.

In order to deduce \emph{convergence of order $\gamma$ globally} from \eqref{eq:loce} we need to analyze the stability of the integration scheme \eqref{eq:numZ}. In the following we set
$$
m_s^n = \max_{0 \leq k \leq n} \{ \Vert E^k\Vert_{s+2}+\Vert F^k\Vert_s + \Vert u^k\Vert_{s+1}\}.
$$
\emph{(i) Error in $F$:} Note that for all $s \in \mathbb{R}$
\begin{equation}\label{stabF}
\Vert \mathrm{e}^{i \tau \Delta}\Vert_s \leq 1,\quad \Vert (i \tau\Delta)^{-1}(1- \mathrm{e}^{i \tau\Delta})\Vert_s \leq 2. 
\end{equation}
Plugging the stability bound \eqref{stabF} into the error recursion \eqref{ErrorrecTaylor-f} for $F$ 
yields that
\begin{equation}\label{errF}
\begin{split}
&\Vert F(t_{n+1}) - F^{n+1}\Vert_s \leq \big(1+ \tau c t_n m_s(t_n) \big) \max_{0\leq k \leq n} \Vert F(t_k)-F^k\Vert_s \\
&\qquad +  c (m_s(0) +t_nm_s^n) \left( \tau\Vert u(t_n)-u^n\Vert_s +\tau \Vert u'(t_n)-u'^n\Vert_s \right) + \Vert L_F^n\Vert_s.
\end{split}
\end{equation}

\emph{(ii) Error in $(\nabs u,u')$:}
We define the operator
\begin{equation}\label{eq:o}
O_\tau = \begin{pmatrix}
\coso & \sino\frac{\nabs}{\nabo} \\[1ex] - \sino\frac{\nabo}{\nabs}  & \coso
\end{pmatrix}.
\end{equation}
Formulas \eqref{ErrorrecTaylor-uprime} and \eqref{ErrorrecTaylor-u} imply that
\begin{equation}\label{erruup1}
\begin{aligned}
& \begin{pmatrix}
\nabs( u(t_{n+1}) - u^{n+1})\\[1ex] u'(t_{n+1}) - u'^{n+1}
\end{pmatrix}  =
O_\tau
\begin{pmatrix}
\nabs( u(t_{n}) - u^{n})\\[1ex] u'(t_{n}) - u'^{n}
\end{pmatrix}
\\& \qquad \qquad + \tau \begin{pmatrix}
\frac{1-\coso}{\tau\nabo} \frac{\nabs}{\nabo}  \\[1ex] \frac{\sino}{\tau\nabo} \end{pmatrix}
\Delta \left( \vert E(t_n)\vert^2 - \vert E^n\vert^2\right) + \begin{pmatrix}
\nabs L_u^n \\[1ex] L_{u'}^n\end{pmatrix}.
\end{aligned}
\end{equation}
Note that the error recursion in $E$ given in \eqref{ErrorrecTaylor-e} yields that
\begin{equation}\label{errE}
\begin{aligned}
\Vert E(t_{n}) - E^{n}\Vert_{s+2}& \leq  (1+c t_n m_s(t_n)) \max_{0\leq k \leq n} \Vert F(t_k)-F^k\Vert_s \\&+ c \big(m_s(0)+ t_nm_s^n\big) \Vert u(t_n)-u^n\Vert_s  + \Vert \Delta L_E^{n-1}\Vert_s,
\end{aligned}
\end{equation}
which allows us to solve the error recursion in $(\nabs u,u')$ as no loss of derivative occurs. More precisely, we have by \eqref{ErrorrecTaylor-e} that
\begin{equation}\label{nabE}
\begin{aligned}
 \nabla^\alpha (E(t_n) - E^n)  ={}& \frac{\nabla^\alpha}{\nabo^2+1} \Big(i( F(t_{n}) - F^{n}) - (u(t_{n})-u^{n})\eta(t_n) \\& + (1-u(t_n))\big( \tau \sum_{k=1}^n (F(t_{k})-F^{k})\big) + \Delta L_E^{n-1}\Big)\\
=&  \frac{\nabla^\alpha}{\nabo^2+1} \eta(t_n) \nabs^{-1} \big( \nabs ( u(t_n) - u^n)\big)+r^\alpha_{\tau,n}
\end{aligned}
\end{equation}
where, for any $\alpha=0,1,2$, $\Vert  \frac{\nabla^\alpha}{\nabo^2+1} \Vert_{s} \leq 1$ and
\begin{equation}\label{bbu}
\begin{aligned}
& \Vert r^\alpha_{\tau,n}\Vert_s \leq   (1+c t_n m_s(t_n)) \max_{0\leq k \leq n} \Vert F(t_k)-F^k\Vert_s + \Vert \Delta L_E^{n-1}\Vert_s\\
&\Vert \eta(t_n)\Vert_s = \Vert E(0) + \tau \sum_{j=1}^n F^j\Vert_s \leq m_s(0)+ t_n m_s^n.
\end{aligned}
\end{equation}
Note that for all $0<\tau\leq 1$ we have
\begin{equation}\label{sincosb}
\left \Vert  \frac{\sino}{\tau\nabo}\right\Vert_s \leq 1,\quad\left \Vert  \frac{1-\coso}{\tau\nabo} \frac{\nabs}{\nabo}\right\Vert_s \leq 2
\end{equation}
and
\begin{equation}\label{deltaE}
\begin{aligned}
\Delta \left( \vert E(t_n)\vert^2 - \vert E^n\vert^2\right)  =& \mathrm{Re}\Big\{ (\overline{E(t_n)+E^n}) \Delta (E(t_n)- E^n) \\&+ (\overline{\Delta E(t_n)+ \Delta E^n} ) (E(t_n) - E^n)\\
& +2 \nabla (\overline{ E(t_n)+ E^n}) \cdot \nabla (E(t_n)-E^n) \Big\}.
\end{aligned}
\end{equation}
In particular,
\begin{equation*}
\begin{aligned}
\Vert \Delta \left( \vert E(t_n)\vert^2 - \vert E^n\vert^2\right) \Vert_s &\leq 4 c \Vert E(t_n) + E^n\Vert_{s+2}\Vert E(t_n)-E^n\Vert_{s+2}\\&\leq 4 c \big(m_s(t_n) + m_s^n\big) \Vert E(t_n)-E^n\Vert_{s+2}.
\end{aligned}
\end{equation*}
Plugging \eqref{sincosb}, \eqref{deltaE}, \eqref{nabE} and \eqref{bbu} into \eqref{erruup1} we obtain that
\begin{equation}
\begin{aligned}\label{rec:uup}
\begin{pmatrix}
\nabs( u(t_{n+1}) - u^{n+1})\\[1ex] u'(t_{n+1}) - u'^{n+1}
\end{pmatrix} & = (O_\tau+P_{\tau,n})\begin{pmatrix}
\nabs( u(t_{n}) - u^{n})\\[1ex] u'(t_{n}) - u'^{n}
\end{pmatrix} +\mathcal{R}_{\tau,n},
\end{aligned}
\end{equation}
where
\begin{equation}\label{eq:p}
P_{\tau,k} = \tau \begin{pmatrix}
 \frac{1-\coso}{\tau\nabo} \frac{\nabs}{\nabo}&
0 \\
\frac{\sino}{\tau\nabo}
&
0
\end{pmatrix}\mathrm{Re}\{  p_{k}^1+p_k^2+p_k^3\},
\end{equation}
with the operators
\begin{equation}\label{def:p}
\begin{aligned}
& p_{k}^1 =   (\overline{E(t_k) +E^k}) \Delta(\nabo^2+1)^{-1} \eta(t_k)  \nabs^{-1},\\
& p_{k}^2=  (\overline{\Delta E(t_k) +\Delta E^k})(\nabo^2+1)^{-1} \eta(t_k)\nabs^{-1},\\
& p_{k}^3 =  2( \overline{\nabla E(t_k) +\nabla E^k})\cdot \nabla(\nabo^2+1)^{-1}   \eta(t_k) \nabs^{-1}.
\end{aligned}
\end{equation}
The remainder satisfies
\begin{equation}\label{restuup}
\begin{aligned}
& \Vert \mathcal{R}_{\tau,n} \Vert_s \leq c \Vert E(t_n) + E^n\Vert_{s+2} \Vert r_{\tau,n}^2\Vert_s+\Vert \nabs L_u^n\Vert_s + \Vert L_{u'}^n\Vert_s
\\& \qquad \leq   c \tau (m_{s}(t_n)+ m_s^n)\big((1+  t_nm_s(t_n) )\max_{0\leq k \leq n} \Vert F(t_k) - F^k\Vert_s  + \Vert \Delta L_E^{n-1}\Vert_s \big)\\&\qquad\qquad+\Vert \nabs L_u^n\Vert_s + \Vert L_{u'}^n\Vert_s.
\end{aligned}
\end{equation}
Note that \eqref{bbu} implies that for all $f\in H^{s-1}(\mathbb{T}^d)$ and $j=1,2,3$ we have
\begin{equation}\label{eq:eta}
\begin{aligned}
 \max_{0\leq k\leq n} \Vert p_{k}^j f\Vert_s &
 \leq c \max_{0\leq k\leq n}\Vert E(t_k) + E^k\Vert_{s+2}  \Vert \eta(t_k)\Vert_s \Vert f\Vert_{s-1}
 \\&\leq c (m_s(t_n)+m_s^n) (m_s(0)+t_n m_s^n)\Vert f \Vert_{s-1}.
 \end{aligned}
\end{equation}
Thus, plugging \eqref{sincosb} and \eqref{eq:eta} into \eqref{eq:o} we obtain that
\begin{equation}\label{boundP}
\begin{aligned}
\max_{0\leq k\leq n}\sup_{0<\tau\leq \tau_0}& \|\tau^{-1} P_{\tau,k}\|_{\mathcal{L}((H^s(\T^d))^2)} \leq  c (1+t_n)(m_s(t_n)+ m_s^n)^2=:q_{n},
\end{aligned}
\end{equation}
The bound \eqref{boundP} and Lemma \ref{lem:stab} below
yield the essential stability bound
\begin{equation}\label{stabWave}
\big \Vert \prod_{k=k_0}^n (O_\tau + P_{\tau,k}) \big \Vert_{\mathcal{L}((H^s(\T^d))^2)} \leq \mathrm{e}^{t_n \big(1+q_{n}\big)},
\end{equation}
for any $1\leq k_0\leq n$.
Thus, solving the error recursion in \eqref{rec:uup} we obtain by the stability bound \eqref{stabWave} and the bound on $\mathcal{R}_{\tau,k}$ in \eqref{restuup} that
\begin{equation}\label{buup}
\begin{split}
&\Vert u(t_{n+1}) -  u^{n+1}\Vert_{s+1} + \Vert
u'(t_{n+1})-u'^{n+1}\Vert_s \\
\leq{} & 2 n \max_{0\leq k \leq n} \Vert \mathcal{R}_{\tau,k}\Vert_s
\max_{1\leq k_0\leq n}\big \Vert \prod_{k=k_0}^n (O_\tau + P_{\tau,k}) \big \Vert_{\mathcal{L}((H^s(\T^d))^2)}\\
\leq{}& \Big(c t_n \big( m_s(t_n) + m_s^n\big)\big\{(1+t_n m_s(t_n)) \max_{0\leq k \leq n } \Vert F(t_k) - F^k\Vert_s \\&+  \max_{0\leq k \leq n}\Vert \Delta L_E^k\Vert_s\big\}
 + 2n \max_{0\leq k \leq n}(\Vert \nabs L_u^k\Vert_s + \Vert L_{u'}^k\Vert_s) \Big) \mathrm{e}^{t_n \big(1+q_{n}\big)}.
\end{split}
\end{equation}
The error bounds \eqref{errF}, \eqref{errE} and \eqref{buup} together with the bound on the local errors in \eqref{eq:loce} yield
\begin{equation}
\Vert F(t_{n+1}) - F^{n+1}\Vert_s \leq \big(1+ \tau A_1(m_s^n) \big)
\max_{0\leq k \leq n} \Vert F(t_k)-F^k\Vert_s + \tau^{1+\gamma}
A_2\label{eq:f-b}
\end{equation}
and
\begin{equation}
\Vert E(t_{n+1}) - E^{n+1}\Vert_{s+2} \leq A_1 (m_s^{n+1})\max_{0\leq
  k \leq n+1} \Vert F(t_k)-F^k\Vert_s + \tau^{\gamma}
A_2,\label{eq:e-b}
\end{equation}
and
\begin{equation}\label{eq:u-b}
\begin{split}
&\Vert u(t_{n+1}) - u^{n+1}\Vert_{s+1} + \Vert
u'(t_{n+1})-u'^{n+1}\Vert_s\\
\leq{} & A_1(m_s^n)  \max_{0\leq k \leq n} \Vert F(t_k)-F^k\Vert_s+ \tau^{\gamma} A_2,
\end{split}
\end{equation}
where $A_1=A_1(\cdot)$ is a continuous and monotonically increasing
function which also depends on $m_{s}(T)$, $A_2$ is a constant which depends on $m_{s+2\gamma}(T)$, and both $A_1$ and $A_2$ depend on $T$, $d$ and $s$.
From \eqref{eq:f-b} we obtain
\begin{equation}\label{eq:f-fb}
\max_{0\leq k \leq n+1} \Vert F(t_{k}) - F^{k}\Vert_s \leq A_2 \sum_{j=0}^{n}\big(1+ \tau A_1(m_s^n) \big)^j \tau^{1+\gamma}\leq  T e^{T A_1(m_s^n)} A_2\tau^\gamma.
\end{equation}
Then, \eqref{eq:u-b} implies
\begin{equation}\label{eq:u-fb}
\max_{0\leq k \leq n+1} \{\Vert u(t_{k}) - u^{k}\Vert_{s+1} + \Vert  u'(t_{k})-u'^{k}\Vert_s\}
\leq (T e^{T A_1(m_s^n)} +1)A_2\tau^\gamma.
\end{equation}
Similarly, \eqref{eq:e-b} implies
\begin{equation}\label{eq:e-fb}
\max_{0\leq k \leq n+1}\Vert E(t_{k}) - E^{k}\Vert_{s+2} 
\leq ( A_1(m_s^{n+1}) T e^{T A_1(m_s^{n+1})}+1)A_2 \tau^\gamma.
\end{equation}
Now, the assertion follows by a continuity argument:
We obtain that
\[
m_s^{n+1} \leq 2( A_1(m_s^{n+1}) T e^{T A_1(m_s^{n+1})}+1)A_2\tau^\gamma+m_s(T)
\]
The quantity $m_s^{n+1}$
depends continuously on $\tau$ and tends to zero as $\tau\to 0$. We conclude that
$m_s^{n+1}\leq 2m_s(T)$ as long as
\[
0<\tau\leq m_s(T)^{\frac{1}{\gamma}} (2( A_1(2m_s(T)) T e^{T A_1(2m_s(T))}+1)A_2)^{-\frac{1}{\gamma}}=:\tau_0.
\]
The claimed estimate (for $n+1$) follows with the constants $c_2=4TA_1(2m_s(T))A_2$ and $c_1=T A_1(2m_s(T))$.
\end{proof}
\begin{lemma}[Local error]\label{lem:locerror}
Let $s>d/2$. For $0\leq \gamma \leq 1$ the local errors defined in \eqref{eq:locErr} satisfy
$$
\max_{0\leq k \leq n} \big\{ \Vert L_F^k\Vert_s + \Vert \nabs L_u^k\Vert_s + \Vert L_{u'}^k\Vert_s + \tau \Vert \Delta L_E^k\Vert_s\big \} \leq c \tau^{1+\gamma} t_n(1+m_{s+2\gamma}(T))^4,
$$
where $m_{s+2\gamma}(T)$ is defined in \eqref{regTaylor} and $c$ depends on $d$ and $s$.
\end{lemma}
\begin{proof}
In the following fix $0\leq \gamma \leq 1$ and let $c$ denote a constant
depending on $s$ and $d$ only.

The mild formulations \eqref{mildZo} and \eqref{eq:solZ} yield
\begin{equation}\label{mdiff}
\begin{split}
&E(t_n+\xi) - E(t_n) \\
=& \frac{e^{i \xi \Delta}-1}{(-
  \xi\Delta)^\gamma}(- \xi\Delta)^\gamma E(t_n) - i \int_0^\xi e^{i (\xi-\lambda)\Delta}  (uE)(t_n+\lambda) \dd \lambda,\\
& F(t_n+\xi) - F(t_n) \\
=& \frac{e^{i \xi \Delta}-1}{(-
   \xi\Delta)^\gamma} (- \xi\Delta)^\gamma F(t_n)- i \int_0^\xi e^{i
   (\xi-\lambda)\Delta}  (uF + u'E)(t_n+\lambda)\dd \lambda,
\end{split}
\end{equation}
and
\begin{equation}\label{mdiff2}
\begin{split}
&\nabo( u(t_n+\xi) - u(t_n)) \\
=& \frac{\cos(\xi\nabo) -1}{ (\xi \nabo)^\gamma} \xi^\gamma \nabo^{1+\gamma} u(t_n) +\frac{\sin(\xi\nabo)}{(\xi \nabo)^\gamma} (\xi \nabo)^{\gamma} u'(t_n) \\&+ \int_0^\xi \sin((\xi-\lambda)\nabo) \Delta \vert E(t_n+\xi)\vert^2,\\
& u'(t_n+\xi) - u'(t_n) \\ =&
  \frac{\cos(\xi\nabo) -1}{ (\xi \nabo)^\gamma} (\xi \nabo)^\gamma u'(t_n)
  - \frac{\sin(\xi\nabo)}{(\xi \nabo)^\gamma}  \xi^\gamma \nabo^{1+\gamma}u(t_n)
  \\&+ \int_0^\xi \cos((\xi-\lambda)\nabo) \Delta \vert E(t_n+\xi)\vert^2.
\end{split}
\end{equation}
Note that
\begin{equation}\label{sincosb2}
\left \Vert  \frac{\sin(\xi\nabo)}{(\xi\nabo)^\gamma}\right\Vert_s \leq 1,\quad\left \Vert  \frac{1-\cos(\xi\nabo)}{(\xi\nabo)^\gamma} \right\Vert_s \leq 2, \quad\left \Vert \frac{1-e^{i\xi\Delta}}{(-\xi \Delta)^\gamma} \right \Vert_s\leq 2.
\end{equation}
Plugging \eqref{sincosb2} into \eqref{mdiff2} yields for $0\leq \xi \leq 1$
\begin{equation}\label{diffuup}
\begin{aligned}
 &\Vert u(t_n+\xi)-u(t_n)\Vert_{s+1}+ \Vert u'(t_n+\xi)-u'(t_n)\Vert_s \\
 &\leq c \xi^\gamma (\Vert u(t_n)\Vert_{s+1+\gamma} + \Vert u'(t_n)\Vert_{s+\gamma}) + c \xi m_s(T)^2 \\&\leq c \xi^\gamma  (1+m_{s+\gamma}(T))^2.
\end{aligned}
\end{equation}
We have
\begin{equation}\label{fracDF}
\begin{aligned}
\Vert (-\Delta)^\gamma F(t)\Vert_s &= \Vert (-\Delta)^\gamma E'(t)\Vert_s 
\\&\leq \Vert (-\Delta)^{\gamma+1} E(t)\Vert_s +c \Vert (-\Delta)^\gamma u(t)\Vert_s \Vert (-\Delta)^\gamma E(t)\Vert_s.
\end{aligned}
\end{equation}
Hence, plugging \eqref{sincosb2}  and \eqref{fracDF} into \eqref{mdiff} yields for $0 \leq \xi \leq 1$
\begin{equation}\label{diffEF}
\begin{aligned}
&\Vert E(t_n+\xi)-E(t_n)\Vert_{s+2}+ \Vert F(t_n+\xi)-F(t_n)\Vert_s\\ &\leq  c \xi^\gamma \Vert E(t_n)\Vert_{s+2+2\gamma}  + c \xi^\gamma (1+m_s(T))^2 + c \xi(1+m_s(T))^3\\
&\leq c \xi^\gamma (1+m_{s+2\gamma}(T))^3.
\end{aligned}
\end{equation}
%Collecting the results in \eqref{diffuup}  and \eqref{diffEF} we obtain by the definition of $m_{s+2\gamma}(T)$ (see \eqref{regTaylor}) that for $0 \leq \xi \leq 1$
%\begin{equation*}\label{taydiff}
%\begin{aligned}
% \Vert E(t_n+\xi)-E(t_n)\Vert_{s+2} +\Vert F(t_n+\xi)-F(t_n)\Vert_s\\+\Vert u(t_n+\xi)-u(t_n)\Vert_{s+1}+ \Vert u'(t_n+\xi)-u'(t_n)\Vert_s\\ \leq c \xi^\gamma (1+m_{s+2\gamma}(T))^3.
%\end{aligned}
%\end{equation*}

\emph{(i) Local errors $\nabs L_u^n$ and $L_{u'}^n$:} By the definition of $\nabs L_u^n$ and $L_{u'}^n$ in \eqref{eq:locErr} we obtain with the aid of \eqref{diffEF} that
\begin{equation}\label{Lu}
\begin{aligned}
&\Vert \nabs L_u^n\Vert_s + \Vert L_{u'}^n\Vert_s \leq c \tau \sup_{0\leq \xi \leq \tau} \Vert \Delta (\vert E(t_n+\xi)\vert^2 - \vert E(t_n)\vert^2)\Vert_s \\
&\leq c \tau m_s(T)  \sup_{0\leq \xi \leq \tau}  \Vert E(t_n+\xi)-E(t_n)\Vert_{s+2} \leq c \tau^{1+\gamma}(1+m_{s+2\gamma}(T))^4.
\end{aligned}
\end{equation}

\emph{(iii) Local errors $L_F^n$ and $\Delta L_E^n$:} By the definition of $\Delta L_E^n$ and $ L_F^n$ in \eqref{eq:locErr} we have
\begin{equation}
\begin{split}\label{llerr}
 \Vert \Delta L_E^n\Vert_s & \leq c m_s(T) \Vert \int_0^{t_n+\tau} F(\lambda) \dd\lambda - \tau \sum_{k=0}^n F(t_{k+1})\Vert_s,\\
 \Vert L_F^n\Vert_s 
& \leq c  \int_0^\tau m_s(T)\Big(\Vert F(t_n+\xi)-F(t_n)\Vert_s + \Vert  \mathcal{F}_{\tau,\xi,n}\Vert_s \Big) \\
&\qquad+ (1+m_s(T))^2\Big( \Vert u(t_n+\xi)-u(t_n)\Vert_s \\
&\quad\qquad+ (1+t_n)\Vert u'(t_n+\xi)-u'(t_n)\Vert_s \Big) \dd \xi,
\end{split}
\end{equation}
cf.\ \eqref{intF} for the definition of $ \mathcal{F}_{\tau,\xi,n}$.
Note that by \eqref{mdiff} we obtain for $0 \leq \xi \leq \tau$ that
\begin{equation*}\label{LF}
\begin{aligned}
  \mathcal{F}_{\tau,\xi,n}&=\int_0^{t_n+\xi}F(\lambda)\dd\lambda  -\tau \sum_{k=0}^n F(t_k) \\&= \sum_{k=0}^{n-1}\int_0^\tau \big(F(t_k+\lambda)-F(t_k)\big)\dd\lambda + \int_0^\xi F(t_n+\lambda) \dd \lambda - \tau F(t_n)\\&
 = \sum_{k=0}^{n-1}\int_0^\tau \frac{e^{i\lambda \Delta}-1}{(-\lambda \Delta)^{\gamma}}(-\lambda\Delta)^\gamma F(t_k)\dd\lambda +\mathcal{R}_{\tau,n}^F,
\end{aligned}
\end{equation*}
where
\begin{equation}
\begin{aligned}
\Vert \mathcal{R}_{\tau,n}^F\Vert_s & \leq \Vert \sum_{k=0}^{n-1} \int_0^\tau \int_0^\lambda e^{i(\lambda-\xi)\Delta} (uF + u' E)(t_n+\xi)\dd \xi \dd \lambda\Vert_s+ \tau cm_s(T)^2
\\&\leq \tau c t_n (1+m_s(T))^3.
\end{aligned}
\end{equation}
Hence, \eqref{sincosb2} together with \eqref{fracDF} implies that for $0 \leq \xi \leq \tau$
\begin{equation}
\begin{aligned}\label{bbF}
 \Vert  \mathcal{F}_{\tau,\xi,n}\Vert_s &\leq c \tau^\gamma t_n m_{s+2\gamma}(T)+c \tau^\gamma  t_n (1+m_s(T))^3.
 \end{aligned}
\end{equation}
Plugging \eqref{diffuup}, \eqref{diffEF} and \eqref{bbF} into \eqref{llerr} we obtain that
\begin{equation}\label{LF-2}
\tau \Vert \Delta L_E^n \Vert_s + \Vert L_F^n\Vert_s \leq c \tau^{1+\gamma}t_n(1+m_{s+2\gamma}(T))^4.
\end{equation}
Collecting the results in \eqref{Lu} and \eqref{LF-2} yields the assertion.
\end{proof}
Recall the definition of the operator $O_\tau$ from \eqref{eq:o}.
\begin{lemma}[Stability lemma]\label{lem:stab} Let $s \in \mathbb{R}$.
For $0<\tau\leq \tau_0$, $1\leq k_0\leq k\leq n$ let $P_{\tau,k}\in \mathcal{L}((H^s(\T^d)^2)$ such that
\[q:=\max_{1\leq k\leq n}\sup_{0<\tau\leq \tau_0}\|\tau^{-1} P_{\tau,k}\|_{\mathcal{L}((H^s(\T^d))^2)}<\infty.\]
Then, for all $(f,g) \in (H^s(\T^d))^2$
\[
\Big\|\prod_{k=k_0}^n (O_{\tau}+P_{\tau,k})\begin{pmatrix}f\\g\end{pmatrix} \Big\|_s\leq e^{n \tau (1+q)}\Big\|\begin{pmatrix}f\\g\end{pmatrix}\Big\|_s.
\]
\end{lemma}
\begin{proof}
Let
\[
V=\frac{1}{\sqrt{2}}\begin{pmatrix}
1& i\\
i&1
\end{pmatrix}, \; Z_\tau:=\begin{pmatrix}
0& \tau \frac{\sin(\tau \nabo)}{\tau \nabo}(\nabs -\nabo)\\
\frac{\sin(\tau \nabo)}{\tau \nabs}(\nabs -\nabo)&0
\end{pmatrix}
\]
We have
\[
O_\tau =V^{-1}\mathrm{diag}(e^{i\tau \nabo},e^{-i\tau \nabo}) V + Z_\tau.
\]
Note that the action of $Z_\tau$ is nothing but multiplication  by $\tau$ of the zero mode of the second component. For $Q_{\tau,k}=V(Z_\tau+P_{\tau,k}) V^{-1}$ we obtain
\[
\prod_{k=k_0}^n (O_{\tau}+P_{\tau,k}) =
V^{-1}\Big\{\prod_{k=k_0}^n\Big(\mathrm{diag}(e^{i\tau \nabo},e^{-i\tau \nabo})+Q_{\tau,k}\Big)\Big\} V.
\]
Hence,
\begin{equation}\label{eq:ub}
\Big\|\prod_{k=k_0}^n (O_{\tau}+P_{\tau,k})\begin{pmatrix}f\\g\end{pmatrix} \Big\|_s \leq \prod_{k=k_0}^n
\Big(1+\|Q_{\tau,k}\|_{\mathcal{L}((H^s(\T^d))^2)}\Big)\Big\|\begin{pmatrix}f\\g\end{pmatrix} \Big\|_s
\end{equation}
Due to
\[
\|Q_{\tau,k}\|_{\mathcal{L}((H^s(\T^d))^2)}\leq \|Z_\tau +P_{\tau,k}\|_{\mathcal{L}((H^s(\T^d))^2)}\leq \tau+\tau q
\]
we conclude that
\[
\prod_{k=k_0}^n \Big(1+\|Q_{\tau,k}\|_{\mathcal{L}((H^s(\T^d))^2)}\Big)\leq \Big(1+\tau(1+q)\Big)^n\leq  e^{n \tau (1+q)},
\]
and the claim follows from \eqref{eq:ub}.
\end{proof}
\subsection{Error analysis for strong solutions and in the energy space}\label{sec:SE}
\begin{remark}\label{rem:convSEhreg}
As lower order Sobolev norms are controlled by higher order Sobolev norms Theorem \ref{thm:s} also yields a convergence result for strong solutions (i.e., in $H^2(\mathbb{T}^d)\times H^1(\mathbb{T}^d)\times L^2(\mathbb{T}^d)$) as well as in the energy space (i.e., in $H^1(\mathbb{T}^d)\times L^2(\mathbb{T}^d)\times H^{-1}(\mathbb{T}^d)$). More precisely, assume that for some $\gamma > 0$ the regularity assumptions \eqref{regTaylor}  hold with $s = d/2+\varepsilon$ for any $\varepsilon>0$. Then there exists a $\tau_0>0$ such that for all $0\leq \tau\leq \tau_0$ and $t_n = n\tau \leq T$ the following convergence bounds hold:
\begin{equation}\label{conv2o}
\|(E(t_n)-E^n, u(t_n)-u^n,u'(t_n)-u'^n)\|_{[r]} \leq c \tau^\gamma   ,\quad r = -1,0.
\end{equation}
However,  the regularity assumptions on the data are quite strong.
\end{remark}

In the following we will show that in dimensions $d \leq 3$ the regularity assumptions \eqref{regTaylor} with $s= \mathrm{max}(1,d/2+\varepsilon)$  actually imply \emph{first-order convergence}, i.e., \eqref{conv2o} holds with $\gamma = 1$. Here, we apply asymmetric product estimates and in order to control the error of $F$ and $u'$ in $L^2(\mathbb{T}^d)$ and $H^{-1}(\mathbb{T}^d)$, respectively, we need a priori bounds on the numerical solutions in higher order Sobolev spaces, cf.\ \cite{Gau15,Lubich08}.

We will carry out the error analysis in detail only for the energy space as the result for strong solutions follows along the same lines. Furthermore, for the sake of clarity of the exposition, we restrict ourselves to dimensions $d\leq 3$, where the following product estimates are crucial for our analysis: For $s_1+s_2\geq 0$ and $1\leq d \leq 3$ we have
\begin{equation}\label{negbiest}
\begin{aligned}
& \Vert f g \Vert_{s} \leq c \Vert f\Vert_{s_1} \Vert g \Vert_{s_2} \quad \text{ for  all } s \leq s_1+s_2-\textstyle\frac{d}{2}\quad \text{ with } s_1,s_2 \text{ and } -s \neq\textstyle \frac{d}{2}\\
&\Vert f g \Vert_{s} \leq c \Vert f\Vert_{s_1} \Vert g \Vert_{s_2} \quad \text{ for  all } s < s_1+s_2-\textstyle\frac{d}{2}\quad \text{ with } s_1,s_2 \text{ or } -s =\textstyle \frac{d}{2}
\end{aligned}
\end{equation}
such that in particular we obtain for $1\leq d\leq 3$ and $\varepsilon>0$ that
\begin{equation}\label{hm1E}
\Vert f g \Vert_{-1}\leq c \Vert f \Vert_{-1} \Vert g \Vert_{\mathrm{max}(d/2+\varepsilon,1)}.
\end{equation}

\begin{theorem}\label{thm:e}
Fix $1\leq d \leq 3$ and $\gamma > 0$. For any $T\in (0,\infty)$ and any $\varepsilon>0$, suppose that for $\delta := \mathrm{max}(d/2+\varepsilon,1)$
\begin{equation*}
E \in \mathcal{C}([0,T];H^{2+\delta+2\gamma}(\mathbb{T}^d)),\quad u \in \mathcal{C}([0,T]; H^{1+\delta+2\gamma}(\mathbb{T}^d)) \cap \mathcal{C}^1([0,T]; H^{\delta+2\gamma}(\mathbb{T}^d))
\end{equation*}
is a mild solution of \eqref{eq:ZakO} with
 \begin{equation}\label{rege}
 m_{\delta+2\gamma}(T) := \sup_{t\in [0,T]} \|(E(t),u(t),u'(t))\|_{[\delta+2\gamma]}<\infty.
 \end{equation}
Then, there exists $\tau_0>0$ such that for all $0\leq \tau\leq \tau_0$ and $t_n = n\tau \leq T$ the trigonometric time-integration scheme \eqref{eq:numZ} is first-order convergent in the energy space, i.e.,
$$ 
\|(E(t_n)-E^n, u(t_n)-u^n,u'(t_n)-u'^n)\|_{[-1]} \leq c\tau ,
$$
where $c$ depends only on $m_\delta(T)$, $T$ and  $d$.
\end{theorem}
\begin{proof} In this proof we proceed similarly to the proof of Theorem \ref{thm:s}. However, we need to be more careful when estimating the nonlinear terms. In the following fix $1 \leq d \leq 3$, $\varepsilon,\gamma >0$ and set $\delta = \mathrm{max}(d/2+\varepsilon,1)$. First note that the regularity assumptions \eqref{rege} together with Theorem \ref{thm:s} (choosing $s=\delta$) imply that there exists a $\tau_0>0$ such that for all $0 \leq \tau \leq \tau_0$ and $t_n = n\tau \leq T$ we have
\begin{equation}\label{numb}
m_\delta^n := \max_{0\leq k \leq n}\{\Vert E^k\Vert_{2+\delta} + \Vert F^k\Vert_\delta+ \Vert u^k\Vert_{1+\delta} + \Vert u'^k\Vert_\delta\} \leq 2 m_{\delta}(T)< \infty.
\end{equation}

In the following we assume that $\tau \leq \tau_0$ such that \eqref{numb}  holds. Furthermore, we denote by $c$ a constant depending only on $m_{\delta}(T)$, $T$, $d$ and prove the claim for $n + 1$ instead of $n$.

The regularity assumptions \eqref{rege} imply that the local errors defined in \eqref{eq:locErr} satisfy
\begin{equation}\label{localErrE}
\max_{0\leq k \leq n} \big\{ \Vert L_F^k\Vert_{-1} + \Vert \nabs L_u^k\Vert_{-1} + \Vert L_{u'}^k\Vert_{-1} + \tau \Vert \Delta L_E^k\Vert_{-1}\big \} \leq c \tau^{2},
\end{equation}
see Lemma \ref{lem:locerrorE} below. In order to deduce first-order convergence globally from \eqref{localErrE} we need to analyze the stability of the integration scheme \eqref{eq:numZ} in the energy space.

\emph{(i) Error in $F$ in $H^{-1}$:} The error recursion in \eqref{ErrorrecTaylor-f} together with the stability bound \eqref{stabF}, the bilinear estimate \eqref{hm1E} and the local error bound \eqref{localErrE} yields that
\begin{equation*}
\begin{aligned}
&\Vert F(t_{n+1})-F^{n+1}\Vert_{-1} \leq \big (1+ \tau c t_n m_\delta(t_n)\big)\max_{0\leq k \leq n}\Vert F(t_k)-F^k\Vert_{-1} \\
&\qquad+ c \big( m_\delta(0)+ t_n  \max_{0\leq k \leq n} \Vert F^k\Vert_\delta\big) \Vert u'(t_n)-u'^n\Vert_{-1} + c \Vert F^n\Vert_{\delta} \Vert u(t_n)-u^n\Vert_1+ c \tau^2.
\end{aligned}
\end{equation*}
Furthermore, the a priori boundedness of the numerical solutions \eqref{numb}  implies that $ \max_{0\leq k \leq n} \Vert F^k\Vert_\delta \leq 2 m_\delta(t_n)< \infty$. Hence, we obtain that
\begin{equation}\label{errFE}
\begin{aligned}
&\Vert F(t_{n+1})-F^{n+1}\Vert_{-1} \leq\big (1+ \tau c)\max_{0\leq k \leq n}\Vert F(t_k)-F^k\Vert_{-1} \\& \qquad + \tau c  \big(\Vert u'(t_n)-u'^n\Vert_{-1} + \Vert u(t_n)-u^n\Vert_{-1}\big) +c \tau^2.
\end{aligned}
\end{equation}

\emph{(ii) Error in $E$ in $H^1$:} Similarly we obtain by the error recursion \eqref{ErrorrecTaylor-e} together with the bilinear estimate \eqref{hm1E}, the bound on the numerical solutions \eqref{numb} and the local error bound \eqref{localErrE} that
\begin{equation}\label{errEE}
\begin{aligned}
\Vert E(t_{n+1})& -E^{n+1}\Vert_{1} \leq c \max_{0\leq k \leq n} \Vert F(t_{k})-F^k\Vert_{-1} +c \Vert u(t_n)-u^n\Vert_{-1}+ c \tau.\\
\end{aligned}
\end{equation}
\emph{(iii) Error in $(u,u')$ measured in $L^2 \times H^{-1}$:} 
The a priori boundedness of the numerical solutions \eqref{numb} together with the bilinear estimate \eqref{hm1E} implies that for $0 \leq \alpha \leq 2$ we have
\begin{equation*}
\begin{aligned}
&\Vert (\nabo^{2-\alpha} E(t_n) + \nabo^{2-\alpha} E^n)\nabo^\alpha (E(t_n)-E^n)\Vert_{-1} 
\\&\qquad \leq c\Vert E(t_n) + E^n\Vert_{2+\delta} \Vert E(t_n) - E^n\Vert_1
\leq 2c m_{\delta}(t_n) \Vert E(t_n) - E^n\Vert_1.
\end{aligned}
\end{equation*}
Thus, similarly to \eqref{rec:uup}  we obtain that
\begin{equation}
\begin{aligned}\label{rec:uupE}
\begin{pmatrix}
\nabs( u(t_{n+1}) - u^{n+1})\\[1ex] u'(t_{n+1}) - u'^{n+1}
\end{pmatrix} & = (O_\tau+P_{\tau,n})\begin{pmatrix}
\nabs( u(t_{n}) - u^{n})\\[1ex] u'(t_{n}) - u'^{n}
\end{pmatrix} +\mathcal{R}_{\tau,n},
\end{aligned}
\end{equation}
where
\begin{equation}\label{RE}
\Vert \mathcal{R}_{\tau,n} \Vert_{-1} \leq c\tau( \max_{0\leq k \leq n} \Vert F(t_k)-F^k\Vert_{-1} + \Vert \Delta L_E^n\Vert_{-1} )+ \Vert \nabs L_u^n\Vert_{-1} + \Vert L_{u'}^n\Vert_{-1}
\end{equation}
and $O_\tau$, $P_{\tau,n}$ are defined in \eqref{eq:o} and \eqref{eq:p}, respectively. The bilinear estimate \eqref{hm1E} together with the a priori boundedness of the numerical solutions \eqref{numb} and the definition of $\eta(t_k)$ in \eqref{bbu} furthermore implies that
\begin{equation*}
\begin{aligned}
\Vert  (\Delta E(t_k) +\Delta E^k)  \eta(t_k)\Vert_{-1} \leq c\Vert E(t_k)+ E^k\Vert_1 \Vert \eta(t_k)\Vert_\delta \leq 2 c(1+t_n)m_{\delta }^2(T).
\end{aligned}
\end{equation*}
Thus, by the definition of $p_k^j$ in \eqref{def:p} we have for all $f \in H^{-1}(\mathbb{T}^d)$ that
$$
\max_{j=1,2,3} \max_{0\leq k \leq n} \Vert p_k^j f\Vert_{-1} \leq c \Vert f \Vert_{-1}.
$$
Together with \eqref{sincosb} (which holds for all $s \in \mathbb{R}$) this yields by the definition of $P_{\tau,k}$ in \eqref{eq:p} that
$$
\max_{0\leq k \leq n} \sup_{0\leq \tau \leq \tau_0} \Vert \tau^{-1} P_{\tau,k}\Vert_{\mathcal{L}((H^{-1}(\mathbb{T}^d))^2)} \leq c.
$$
Hence, solving the error recursion in \eqref{rec:uupE} we obtain with the aid of the stability Lemma \ref{lem:stab}, the bound on $\mathcal{R}_{\tau,n}$ given in \eqref{RE} together with the local error bound \eqref{localErrE} that
\begin{equation}\label{erruupE}
\begin{split}
&\Vert u(t_{n+1}) - u^{n+1}\Vert_0 + \Vert u'(t_{n+1})-u'^{n+1}\Vert_{-1}\\
&\leq c  \max_{0\leq k \leq n} \Vert F(t_k)-F^k\Vert_{-1} +c \tau.
\end{split}
\end{equation}
Collecting the results in \eqref{errFE}, \eqref{errEE} and \eqref{erruupE} yields the assertion.
\end{proof}

\begin{remark}
Note that in the limit $\tau \to 0$ Theorem \ref{thm:e} (together with Remark \ref{rem:lwp}) implies first order-convergence in the energy space if for some $\varepsilon>0$
\begin{equation*}
\begin{aligned}
& \Vert (E(0),u(0),u'(0))\Vert_{[1+\varepsilon]} < \infty  \quad &\text{for }&d = 1,2,\\
&\Vert (E(0),u(0),u'(0))\Vert_{[3/2+\varepsilon]}< \infty \quad &\text{for }& d = 3.
\end{aligned}
\end{equation*}
\end{remark}
\begin{lemma}[Local error in the energy space]\label{lem:locerrorE}
Let $1\leq d \leq 3$. Then the local errors defined in \eqref{eq:locErr} satisfy for any $\varepsilon>0$ and $\delta \geq \mathrm{max}(d/2+\varepsilon,1)$
$$
\max_{0\leq k \leq n} \big\{ \Vert L_F^k\Vert_{-1} + \Vert \nabs L_u^k\Vert_{-1} + \Vert L_{u'}^k\Vert_{-1} + \tau \Vert \Delta L_E^k\Vert_{-1}\big \} \leq c \tau^{2}(1+m_{\delta}(T))^4,
$$
where $m_\delta(T)$ is defined in \eqref{rege} and $c$ depends on $T$ and $d$.
\end{lemma}
\begin{proof} The strategy of proof is similar to the one of Lemma \ref{lem:locerror}. In the following fix $\varepsilon>0$ and set $\delta = \mathrm{max}(d/2+\varepsilon,1)$. Let $c$ denote a constant depending only on $d$ . The local error representation \eqref{eq:locErr} together with the bilinear estimate \eqref{hm1E} implies that
\begin{equation}\label{locErrE}
\begin{aligned}
&\Vert L_F^n\Vert_{-1}  \leq c  \int_0^\tau m_\delta(T)\Big(\Vert F(t_n+\xi)-F(t_n)\Vert_{-1} + \Vert  \mathcal{F}_{\tau,\xi,n}\Vert_{-1} \Big) \\
&\qquad + (1+m_\delta(T))^2\Big( \Vert u(t_n+\xi)-u(t_n)\Vert_{-1} \\
&\quad \qquad+ (1+t_n)\Vert u'(t_n+\xi)-u'(t_n)\Vert_{-1} \Big) \dd \xi,\\
& \Vert \nabs L_u^n\Vert_{-1}+ \Vert L_{u'}^n\Vert_{-1}\leq cm_\delta(T) \int_0^\tau \Vert E(t_n+\xi)-E(t_n)\Vert_{1} \dd \xi ,\\
& \Vert \Delta L_E^n\Vert_{-1} \leq c m_\delta(T) \Vert \int_0^{t_n+\tau} F(\lambda)\dd \lambda - \tau \sum_{k=0}^n F(t_{k+1})\Vert_{-1}.
\end{aligned}
\end{equation}
Choosing $\gamma = 1$ in \eqref{mdiff2}  we obtain with the aid of \eqref{sincosb2} (which holds for all $s\in \mathbb{R}$) and the bilinear estimate \eqref{hm1E} that
\begin{equation}
\begin{aligned}\label{euup}
&\Vert u(t_n+\xi)-u(t_n)\Vert_0 + \Vert u'(t_n+\xi)-u'(t_n)\Vert_{-1} \\
&\leq c \xi (\Vert u(t_n)\Vert_{1}+ \Vert u'(t_n)\Vert_{0})+c \xi m_\delta(T)^2\\
& \leq c \xi (1+m_\delta(T))^2.
\end{aligned}
\end{equation}
Note that
\begin{equation}\label{fracDF2}
\Vert (-\Delta) F(t)\Vert_{-1} \leq \Vert F(t)\Vert_1 = \Vert E'(t)\Vert_1 \leq \Vert E(t)\Vert_3 + c m_\delta(T)^2.
\end{equation}
Choosing $\gamma = 1$ in \eqref{mdiff} we thus obtain by \eqref{sincosb2}, \eqref{hm1E} and \eqref{fracDF2} that
\begin{equation}\label{eef}
\begin{aligned}
&\Vert E(t_n+\xi)-E(t_n)\Vert_1 + \Vert F(t_n+\xi)-F(t_n)\Vert_{-1} \\
&\leq c \xi \Vert E(t_n)\Vert_3 + c\xi(1+m_\delta(T))^3\\
& \leq c \xi (1+m_\delta(T))^3.
\end{aligned}
\end{equation}
%Collecting the results in \eqref{euup} and \eqref{eef} yields by the definition of $m_\delta(T)$ in \eqref{rege} that
%\begin{equation*}\label{efinal}
%\begin{aligned}
%\Vert E(t_n+\xi)-E(t_n)\Vert_1 + \Vert F(t_n+\xi)-F(t_n)\Vert_{-1}\\
%\Vert u(t_n+\xi)-u(t_n)\Vert_0 + \Vert u'(t_n+\xi)-u'(t_n)\Vert_{-1} \\
%\leq c \xi (1+m_\delta(T))^3.
%\end{aligned}
%\end{equation*} 
Similarly we can show that for $0\leq \xi \leq \tau$ we have (see also \eqref{bbF})
\begin{equation}\label{bbF2}
\begin{aligned}
 \Vert \int_0^{t_n+\xi}F(\lambda)\dd\lambda  -\tau \sum_{k=0}^n F(t_k) \Vert_s &\leq c \tau  t_n (1+m_\delta(T))^3.
 \end{aligned}
\end{equation}
Plugging \eqref{euup}, \eqref{eef} and \eqref{bbF2} into \eqref{locErrE} yields the assertion.
\end{proof}
For strong solutions we obtain the following convergence result:
\begin{theorem}\label{thm:st}
Fix $1\leq d \leq 3$. For any $T\in (0,\infty)$, suppose that
\begin{equation*}
E \in \mathcal{C}([0,T];H^{4}(\mathbb{T}^d)),\quad u \in \mathcal{C}([0,T]; H^{3}(\mathbb{T}^d)) \cap \mathcal{C}^1([0,T]; H^{2}(\mathbb{T}^d))
\end{equation*}
is a mild solution of \eqref{eq:ZakO} with
 \begin{equation}\label{bms}
 m_{2}(T) := \sup_{t\in [0,T]} \|(E(t),u(t),u'(t))\|_{[2]}<\infty.
 \end{equation}
Then, there exists $\tau_0>0$ such that for all $0\leq \tau\leq \tau_0$ and $t_n = n\tau \leq T$ the scheme \eqref{eq:numZ} is first-order convergent in the sense that
$$ 
\|(E(t_n)-E^n, u(t_n)-u^n,u'(t_n)-u'^n)\|_{[0]} \leq c\tau ,
$$
where $c$ depends only on $m_2(T)$, $T$ and  $d$.
\end{theorem}
\begin{proof} Fix $0 <\varepsilon \ll 1$. Note that the regularity assumptions \eqref{bms}  imply that there exists a $\tau_0>0$ such that for all $0 \leq \tau \leq \tau_0$ and $t_n = n \tau \leq T$ the numerical solutions satisfy for $ \delta = d/2+ \varepsilon$ the a priori bound
$$
m^n_\delta := \max_{0\leq k \leq n} \{ \Vert E^k\Vert_{\delta+2} + \Vert F^k\Vert_\delta + \Vert u^k\Vert_{\delta+1} + \Vert u'^k\Vert_{\delta} \} \leq 2 m_{2}(T) < \infty.
$$
This follows from choosing  $s=d/2+\varepsilon$ and $\gamma = 1-s/2 $ in Theorem \ref{thm:s}, whereupon in particular $\gamma = 1-d/4 - \varepsilon/2> 0$ as $d \leq 3$. Together with the $L^2(\mathbb{T}^d)$ estimate
\begin{equation}\label{L2est}
 \Vert f g \Vert_0 \leq c \Vert f\Vert_0 \Vert g \Vert_\delta,
\end{equation}
the proof can be completed along the lines of the proof of Theorem \ref{thm:e}.
\end{proof}

\section{Second-order Scheme}\label{sec:2app}
In this section we derive a second-order trigonometric integration scheme for the Zakharov system \eqref{eq:ZakO} based on the mild solutions \eqref{eq:solZ}. In order to achieve this we use a second-order exponential integrator in the approximation of $F$. Furthermore, we approximate the integrals in $(u,u')$ with a trapezoidal rule, i.e., we use that
\begin{equation}\label{trapez}
\int_0^\tau f(\xi) \dd \xi = \frac{\tau}{2} \left(f(\tau) +
  f(0)\right) + \mathcal{O}_s(\tau^3 \sup_{0 \leq \xi \leq \tau} \Vert f''(\xi) \Vert_s),
\end{equation}
where for notational simplicity we here use the notation
$\mathcal{O}_s(z)$ which denotes a remainder term depending on $z\geq 0$ when measured in $H^s$, i.e.,
$$
f = g + \mathcal{O}_s (z) \quad \text{ if } \quad \Vert f - g \Vert_s \leq c z,
$$
for some constant $c> 0$.

Using the second-order Taylor series expansion
\begin{equation*}\label{T2}
\begin{aligned}
(uF+u' \mathcal{I}_F)(t_n+\xi) &= (uF+u' \mathcal{I}_F)(t_n) + \xi (uF+u' \mathcal{I}_F)'(t_n)\\&+ \mathcal{O}_s\Big( \tau^2 \sup_{0 \leq \xi \leq \tau} \Vert (uF+u' \mathcal{I}_F)''(t_n+\xi)\Vert_s\Big)
\end{aligned}
\end{equation*}
in the integral in $F$ as well as the trapezoidal rule \eqref{trapez} for the approximation of the integrals in $u$ and $u'$ yields the following approximation to the  solutions \eqref{eq:solZ}  of the Zakharov system: For sufficiently smooth solutions we have
\begin{equation}\label{eq:solZT2}
\begin{split}
F(t_n+\tau)= &\mathrm{e}^{i \tau \Delta} F(t_n) -i\int_0^\tau \mathrm{e}^{i (\tau-\xi)\Delta}\dd \xi 
(uF+u' \mathcal{I}_F)(t_n) \\
& -i\int_0^\tau \mathrm{e}^{i (\tau-\xi)\Delta} \xi \dd \xi \big( u' F + uF' +u'' \mathcal{I}_F + u' \mathcal{I}_F'\big)(t_n) \\
& + \mathcal{O}_s\Big(\tau^3 \sup_{0 \leq \xi \leq \tau} \Vert (uF+u'\mathcal{I}_F)''(t_n+ \xi)\Vert_s\Big) \\
 u(t_n+\tau) =& \cos (\tau \nabo)u(t_n) + \nabo^{-1} \sin (\tau \nabo) u'(t_n)+ \frac{\tau}{2} \frac{\sin(\tau\nabo)}{\nabo} \Delta \vert E(t_n)\vert^2 \\&
 + \mathcal{O}_s\Big(\tau^3 \sup_{0 \leq \xi \leq \tau} \Vert \vert E \vert '' (t_n+ \xi)\Vert_{s+1}\Big)\\
  u'(t_n+\tau) =& - \nabo\sin (\tau \nabo)u(t_n) +\cos (\tau \nabo) u'(t_n)\\
&+ \frac{\tau}{2} \left( \Delta \vert E(t_n+\tau)\vert^2 + \mathrm{cos}(\tau \nabo) \Delta \vert E(t_n)\vert^2 \right)\\
& + \mathcal{O}_s\Big(\tau^3 \sup_{0 \leq \xi \leq \tau} \Vert \vert E \vert '' (t_n+ \xi)\Vert_{s+2}\Big)\\
 E(t_n+\tau) =& (1 -\Delta)^{-1}\Big(i F(t_n+\tau)-(u(t_n+\tau)-1)\mathcal{I}_F(t_n+\tau)\Big).
\end{split}
\end{equation}
In order to derive a robust scheme we  integrate the terms involving $\mathrm{e}^{i \xi \Delta}\xi^\delta$ with $\delta = 0,1$ exactly, i.e., we will use that
\begin{equation}\label{intsT2}
\begin{split}
&   \int_0^\tau \mathrm{e}^{i (\tau-\xi)\Delta} \dd \xi  f = -\frac{1}{i \Delta} \left(1 - \mathrm{e}^{i \tau \Delta}\right) f,\\
& \int_0^\tau \mathrm{e}^{i (\tau-\xi)\Delta} \xi \dd \xi  f =-\frac{1}{i \Delta} \big(\tau + \frac{1}{i \Delta}(1-\mathrm{e}^{i \tau \Delta} )\big) f.
\end{split}
\end{equation}
Next we need to derive a suitable approximation to $\mathcal{I}_F$ defined in \eqref{IF}. We have
\begin{equation}\label{IF2a}
\mathcal{I}_F(t_n) = E_0 + \int_0^{t_n} F(\lambda) \mathrm{d}\lambda = E_0 + \sum_{k=0}^{n-1} \int_0^\tau F(t_k+\lambda) \dd \lambda.
\end{equation}
Note that by \eqref{eq:solZ} and \eqref{intsT2} we have
\begin{equation}
\begin{split}
 \int_0^\tau   F(t_k+\lambda)  \dd \lambda  &= \int_0^\tau \mathrm{e}^{i \lambda \Delta} F(t_k)\dd \lambda - i \int_0^\tau \int_0^\lambda \mathrm{e}^{i (\lambda - \xi)\Delta} (uF+u'\mathcal{I}_F)(t_k+\xi)\dd\xi \dd \lambda\\
 &= \int_0^\tau \mathrm{e}^{i \lambda \Delta} F(t_k)\dd \lambda - i \int_0^\tau \int_0^\lambda \mathrm{e}^{i (\lambda - \xi)\Delta} (uF+u'\mathcal{I}_F)(t_k)\dd\xi \dd \lambda\\
 & + \mathcal{O}_s \big( \tau^3 \sup_{0 \leq \xi \leq \tau} \Vert (uF+u'\mathcal{I}_F)'(t_k+\xi)\Vert_s\big)\\
& =  \frac{1}{i \Delta} (\mathrm{e}^{i \tau \Delta}-1) F(t_k) \\
&+ \frac{1}{\Delta} \Big( \tau - \frac{1}{i \Delta} (\mathrm{e}^{i \tau \Delta} -1)\Big) \Big( u(t_k)F(t_k) + u'(t_k) (E_0 + \tau \sum_{j= 0}^k F(t_j)) \Big)\\
& + \mathcal{O}_s \big( \tau^3 \sup_{0 \leq \xi \leq \tau} \Vert (uF+u'\mathcal{I}_F)'(t_k+\xi)\Vert_s \big)
\big).\\
\end{split}
\end{equation}
Plugging the above expansion into \eqref{IF2a}  yields that
\begin{equation}\label{IFT2}
\begin{split}
\mathcal{I}_F(t_n) 
& = E_0 - \tau \mathcal{D}_1(\tau \Delta) \sum_{k= 0}^{n-1} F(t_k) \\
& + \tau \mathcal{D}_2(\tau \Delta)  \sum_{k=0}^{n-1}  \Big( u(t_k)F(t_k) + u'(t_k) (E_0 + \tau \sum_{j= 0}^k F(t_j)) \Big)\\
& +  \mathcal{O}_s \big( \tau^2 t_n \sup_{0 \leq \xi \leq t_n} \Vert (uF+u'\mathcal{I}_F)'(\xi)\Vert_s
\big),\\
\end{split}
\end{equation}
where
\begin{equation}\label{Dop}
\mathcal{D}_1(\tau \Delta) := \frac{1-\mathrm{e}^{i\tau\Delta}}{i\tau \Delta}, \quad \mathcal{D}_2(\tau \Delta) :=\Delta^{-1} \big(1+ \mathcal{D}_1(\tau\Delta)\big).
\end{equation}
Using the differential equations \eqref{eq:Zak} as well as the definition of $\mathcal{I}_F$ in \eqref{IF} we furthermore obtain that
\begin{equation}\label{eqit}
\begin{aligned}
& \big( u' F + uF' +u'' \mathcal{I}_F + u' \mathcal{I}_F'\big)(t_n)  \\
& =\big( u' F + u (i \Delta F - i u F - i u' \mathcal{I}_F)  + \mathcal{I}_F \Delta (u + \vert E\vert^2) + u' F\big)(t_n).
\end{aligned}
\end{equation}
Plugging the relations \eqref{intsT2}, \eqref{IFT2} and \eqref{eqit} into \eqref{eq:solZT2} yields a  second-order trigonometric time-integration scheme by setting
\begin{equation}\label{inT2}
\begin{aligned}
&E^0 = E_0,\quad u^0=u_0,\quad u'^0 = u_1, \\
&F^0 = i (\Delta E^0-u^0E^0),\quad S_F^0 =\tau F^0, \quad \mathcal{I}_F^0 := E_0
\end{aligned}
\end{equation}
and for $n \geq 0$
\begin{equation}\label{eq:numZ2}
\begin{aligned}
 F^{n+1}  &= \mathrm{e}^{i \tau \Delta} F^n +i \tau \mathcal{D}_1(\tau\Delta)
\big( u^n F^n + u'^n \mathcal{I}_F^n)\\
&+\tau  \mathcal{D}_2(\tau \Delta )\Big( 2u'^n F^n +i u^n ( \Delta F^n -  u^n F^n -  u'^n \mathcal{I}_F^n )  + \mathcal{I}_F^n \Delta (u^n + \vert E^n\vert^2) \Big),\\
  u^{n+1}& = \cos (\tau \nabo)u^n + \nabo^{-1} \sin (\tau \nabo) u'^n +\frac{\tau}{2} \nabo^{-1}\sin(\tau\nabo) \Delta \vert E^n\vert^2,\\
    \mathcal{I}_F^{n+1} &= E_0 - \mathcal{D}_1(\tau\Delta)  S_F^n
 + \tau \mathcal{D}_2(\tau\Delta) \sum_{k=0}^{n}  \Big( u^kF^k+ u'^k (E_0 + S_F^k \big)\Big),\\
   E^{n+1}  &=  (-\Delta+1)^{-1}\left(  i F^{n+1} -(u^{n+1}-1)  \mathcal{I}_F^{n+1}\right),\\
    u'^{n+1}  &= - \nabo\sin (\tau \nabo)u^n +\cos (\tau \nabo) u'^n+ \frac{\tau}{2} \left( \Delta \vert E^{n+1}\vert^2 + \mathrm{cos}(\tau \nabo) \Delta \vert E^n\vert^2 \right),\\
  S_F^{n+1}  &= S_F^n + \tau F^{n+1}.
\end{aligned}
\end{equation}

\begin{remark}[Second-order convergence]
For sufficiently smooth solutions the trigonometric integration scheme \eqref{eq:numZ2} is second-order convergent without imposing any spatial-dependent time-step restriction, i.e.,  also in the limit $\Delta x \to 0$. More precisely, Theorem \ref{thm:s} holds for \eqref{eq:numZ2}  with $\gamma = 2$. The ideas in the error analysis are thereby similar to the ones used in Section~\ref{sect:err}. The only additional important estimate is that
\begin{equation*}
\Vert \mathcal{D}_2(\tau \Delta) \big( (\Delta f) g\big)\Vert_s \leq c \Vert f \Vert_s \Vert g \Vert_{s+2}
\end{equation*}
for some constant $c>0$.  We omit the details of the proof and refer to  \cite{Gau15}  for the analysis of second-order trigonometric integrators for semilinear wave equations and to \cite{HochOst10} for the analysis of higher-order exponential integrators.
\end{remark}

\begin{remark}
Note that for given $(E^n,F^n,u^n,u'^n, S_F^n, \mathcal{I}_F^n)$ we can compute the next iteration in \eqref{eq:numZ2} without saving $(E^k,F^k,u^k,u'^k, S_F^k, \mathcal{I}_F^k)$ for any $k < n$ by setting
\begin{equation}\label{eq:numZ2one}
\begin{aligned}
 \mathcal{I}_F^{n+1} & :=  \mathcal{I}_F^n - \tau \mathcal{D}_1(\tau \Delta)F^{n}+ \tau \mathcal{D}_2(\tau \Delta)  \Big( u^{n}F^{n} + u'^{n} (E_0+S_F^{n}) \Big), \quad \mathcal{I}_F^0 = E_0.
\end{aligned}
\end{equation}
\end{remark}

%\begin{comment}
\section{Numerical experiments}
In this section we numerically confirm the first-, respectively, second-order convergence rate of the trigonometric time-integration schemes \eqref{eq:numZ} and \eqref{eq:numZ2} towards the exact solutions of the Zakharov system \eqref{eq:ZakO}. Furthermore, we \emph{numerically test the geometric properties} of the trigonometric integration schemes, i.e., the conservation of the $L^2$ norm of $E(t)$ (see \eqref{eq:l2con}), the conservation of the energy (see \eqref{eq:ham}) as well as the shape preservation of solitary waves over ``long times''.

\begin{remark}
Note that in the derived convergence bounds, the
  error constants depend on $T$, which is natural in subcritical regimes. Nevertheless, the numerical findings suggest that for a sufficiently small CFL number the  geometric quantities are preserved on ``long'' time intervals. Thereby, we define the value CFL~$ : = \tau (\Delta x)^{-2}$ with $\tau$ and $\Delta x$ denoting the time- and spatial-step size, respectively.
\end{remark}

\begin{remark}
In the numerical experiments we use a standard Fourier pseudo-spectral
method for the space discretization. For sufficiently smooth solutions
the fully discrete error then behaves like $\tau + K^{-r}$ for the first-order scheme and like $\tau^2 + K^{-r}$ for the second-order scheme for some $r>0$ depending on the smoothness of the solutions. For a fully discrete analysis of exponential-type time integrators coupled to a spectral approximation in space for Schr\"odinger, respectively, semilinear wave equations we refer to \cite{Faou12}  and \cite{Gau15}, respectively.
\end{remark}

\begin{example}\label{ex:1}
We consider the Zakharov system \eqref{eq:ZakO} set on the one-dimensional torus $\mathbb{T}$ with initial values
\begin{equation}\label{ex1:ini}
\begin{aligned}
& E(0,x) = (2-\cos(x)\sin(2x))^{-1}\sin(2x) \cos(4x)+i \sin(2x) \cos(x),\\
& u(0,x) =  (2-\sin(2x)^2)^{-1}\sin(x) \cos(2x), \quad \partial_t u(0,x) = (2-\cos(2x)^2)^{-1}\sin(x)
\end{aligned}
\end{equation}
normalized in $H^2, H^1$ and $L^2$, respectively. In order to test the convergence rate of the first-, respectively, second-order trigonometric time-integration scheme \eqref{eq:numZ}  and \eqref{eq:numZ2} we take the numerical method presented in \cite{Bao05} as a reference solution. For the latter we choose a very small time-step size to ensure to be sufficiently close to the exact solutions. For the space discretization we choose the largest Fourier mode $K = 2^{10}$ (i.e., the spatial mesh size $\Delta x = 0.0061$) and integrate up to $T = 1$. The error of $(E,u)$ measured in the corresponding discrete $H^2\times H^1$ norm is illustrated in Figure \ref{fig:Example1}.

\begin{figure}[h!]
\centering
\includegraphics[width=0.4\linewidth]{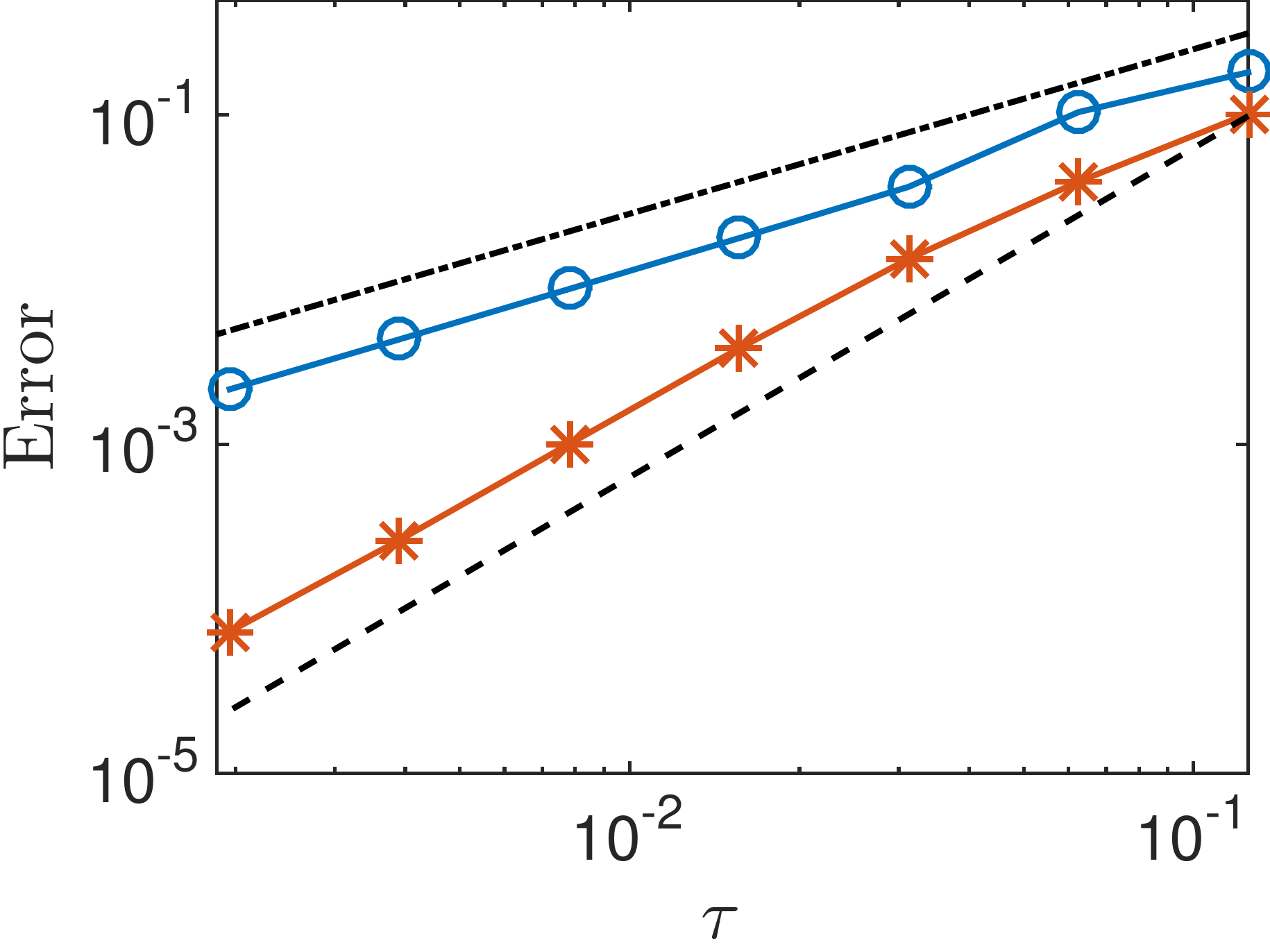}
\hfill
\includegraphics[width=0.4\linewidth]{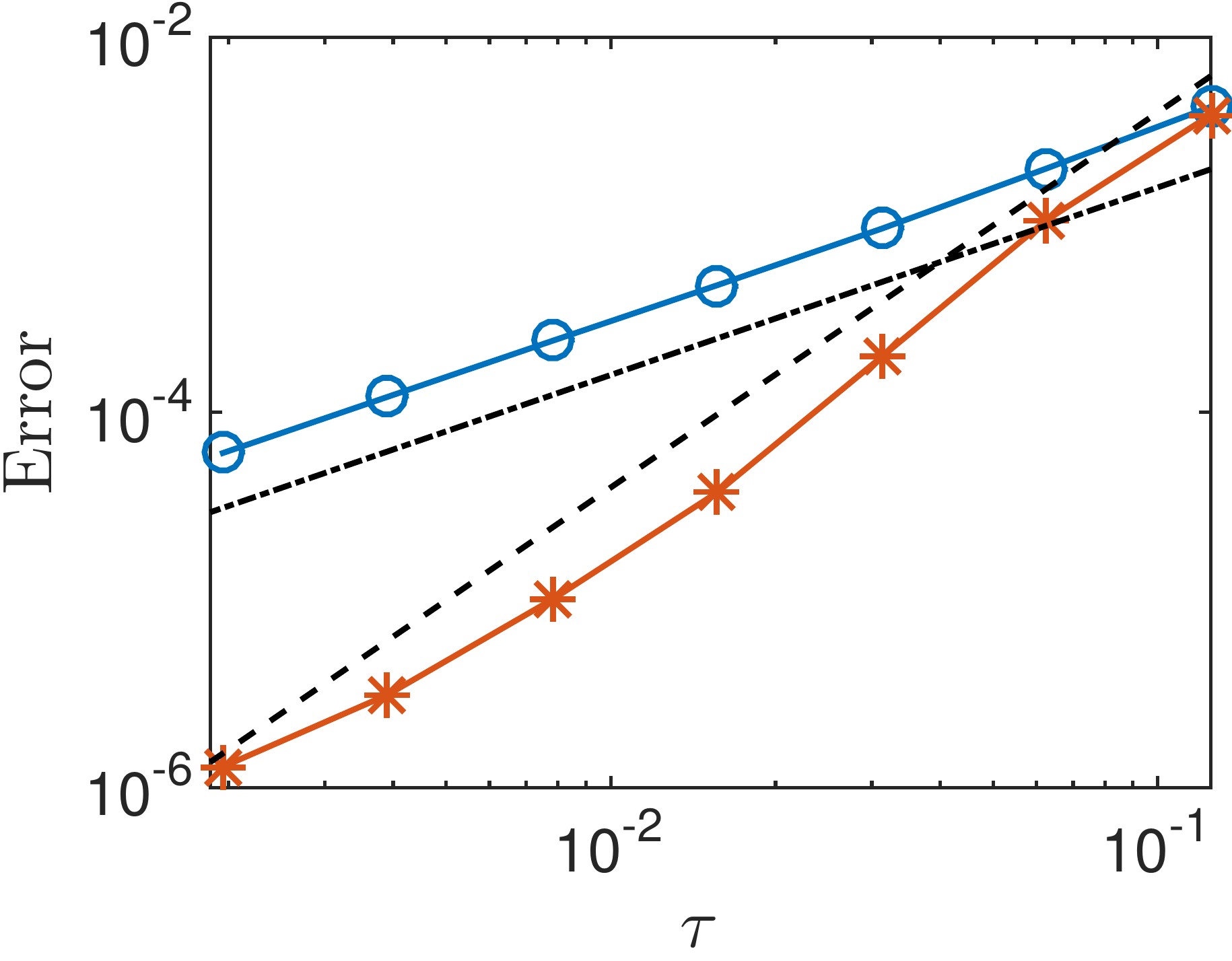}
\caption{Orderplot (double logarithmic). Convergence rates of the first-order scheme \eqref{eq:numZ}  (blue, circle) and the second-order scheme \eqref{eq:numZ2} (red, star). Left picture: Error in $E$ measured in $H^2$. Right picture: Error in $u$ measured in $H^1$. The slope of the dashed-dotted and dashed line is one and two, respectively.}\label{fig:Example1}
\end{figure}

\end{example}

\begin{example}[Solitary waves]\label{exSW}
Exact solutions of the Zakharov system \eqref{eq:ZakO}  are explicitly given by so-called solitary wave solutions, which for the Zakharov system set on $\mathbb{R}$ are described by
\begin{equation}
\begin{aligned}
& E(t,x) = \sqrt{2 B^2(1-C^2)} \mathrm{sech}(B(x-Ct)) \exp\big(i \left(C/2x - \left( C^2/4-B^2\right) t)\right)\big),\\
& u(t,x) = - 2 B^2 \mathrm{sech}^2\left(B(x-Ct)\right),\\&
 \partial_t u(t,x) = - 4B^3 C\mathrm{sinh}\left(B(x-Ct)\right)\mathrm{cosh}^{-3}\left(B(x-Ct)\right)
\end{aligned}\label{solsol}
\end{equation}
with $B,C \in \mathbb{R}$. For the numerical simulations we choose ``a large torus'' (more precisely $x \in [-10 \pi,10\pi]$). In Figure~\ref{fig:exSWS} we simulate the soliton solution \eqref{solsol} with the trigonometric integration schemes~\eqref{eq:numZ} and \eqref{eq:numZ2} up to $T= 100$. We carry out the simulations for two different CFL numbers. Furthermore, we set $B = 0.5$ and $C = 0.15$.

\begin{figure}[h!]
\centering
\includegraphics[width=0.47\linewidth]{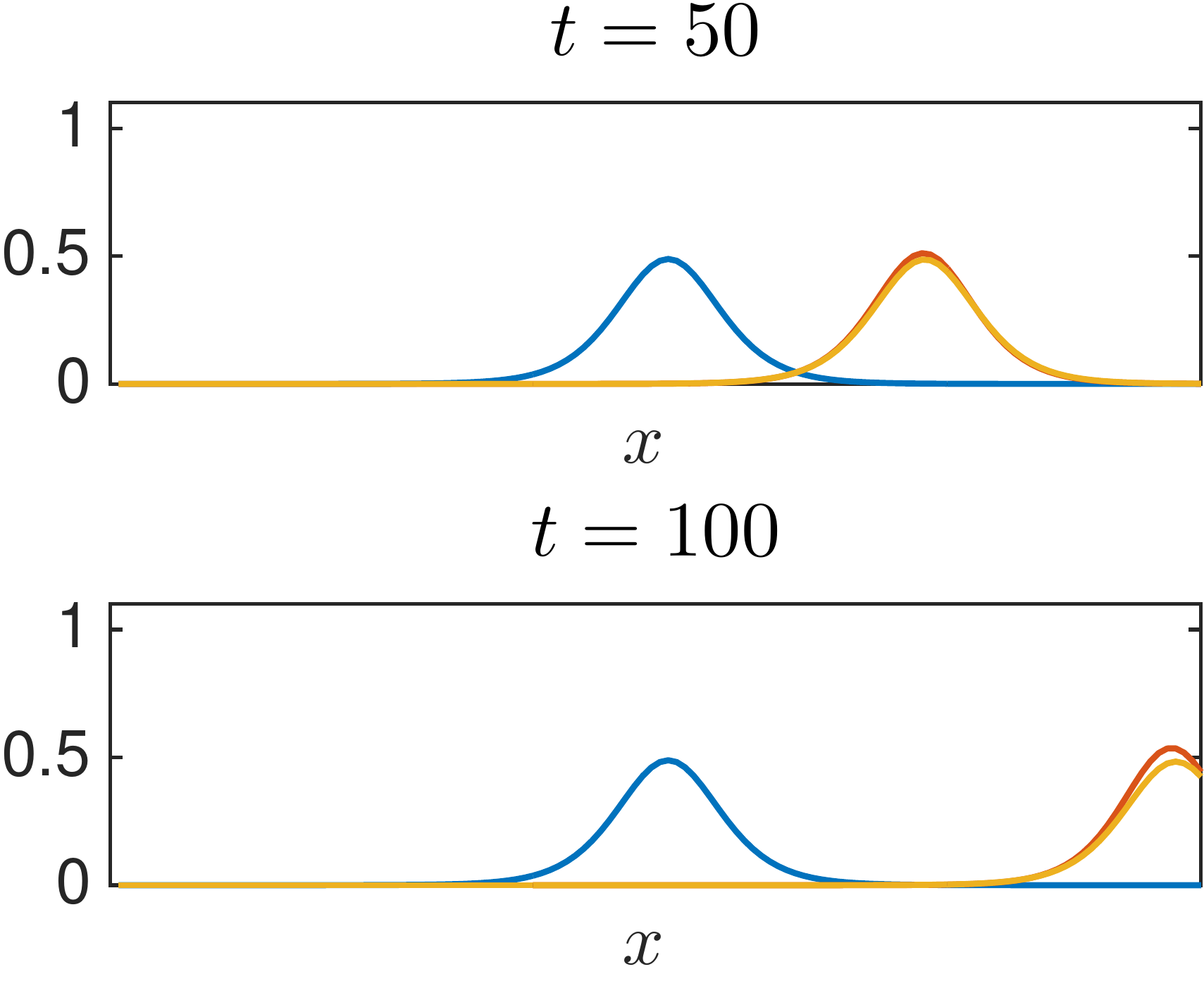}
\hfill
\includegraphics[width=0.47\linewidth]{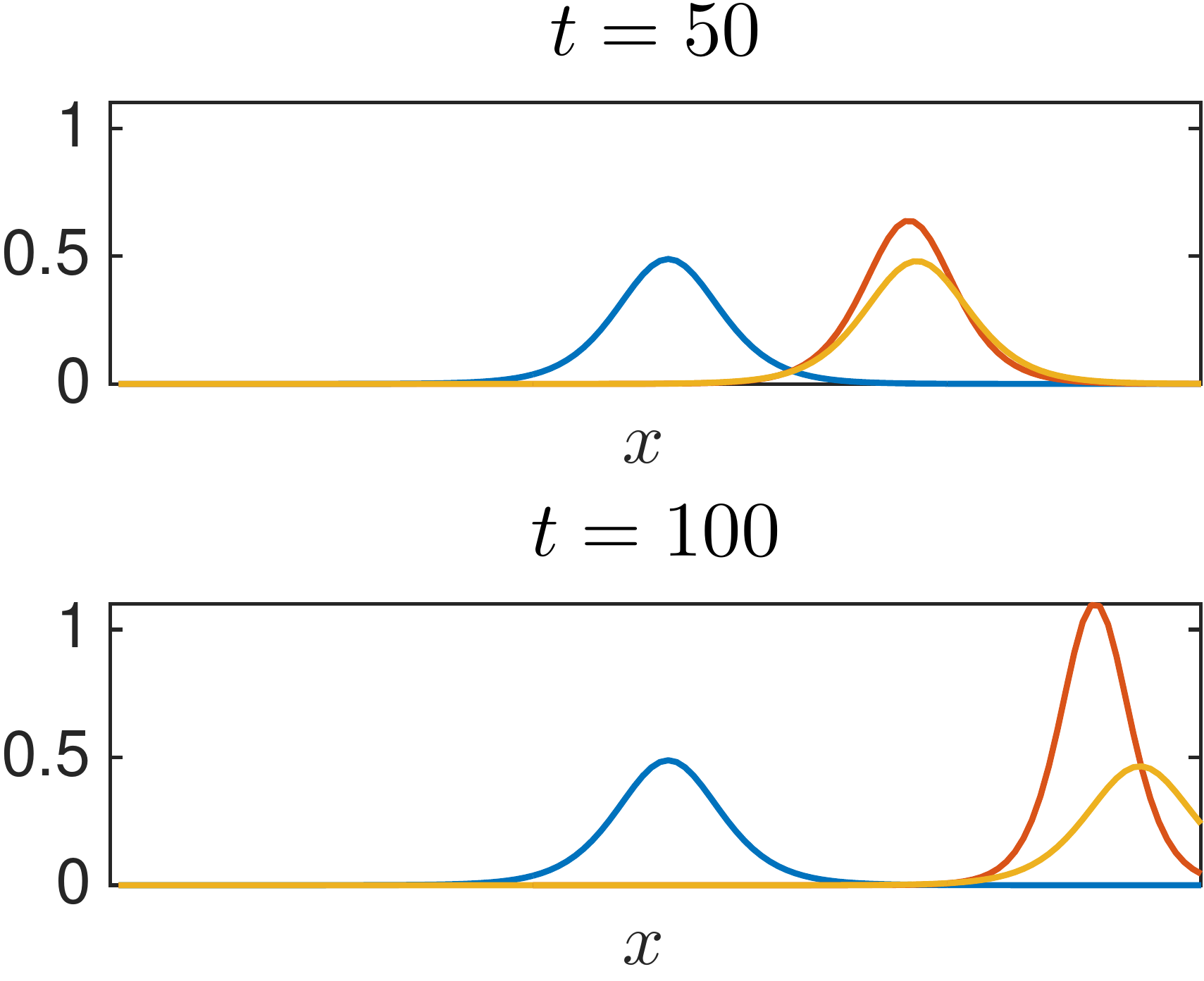}
\caption{Simulation of solitary wave $\vert E^n\vert$ at different times $t$ with first-order scheme \eqref{eq:numZ} (red)  and second-order scheme \eqref{eq:numZ2} (yellow). Initial profile: blue. Left picture: CFL$=3.2$. Right picture: CFL$=32$.}\label{fig:exSWS}
\end{figure}

\end{example}

\begin{example}[Energy conservation]
In this example we numerically test the $L^2$ conservation \eqref{eq:l2con} and the energy conservation \eqref{eq:ham} of the first-and second-order trigonometric time-integration scheme \eqref{eq:numZ} and \eqref{eq:numZ2}, respectively. The numerical findings are illustrated in Figure \ref{fig:Energy} (left picture: first-order scheme, right picture: second-order scheme). In both simulations we choose CFL$ \approx 5$. For a too large CFL number additional numerical findings suggest that the energy is no longer conserved.

\begin{figure}[h!]
\centering
\includegraphics[width=0.43\linewidth]{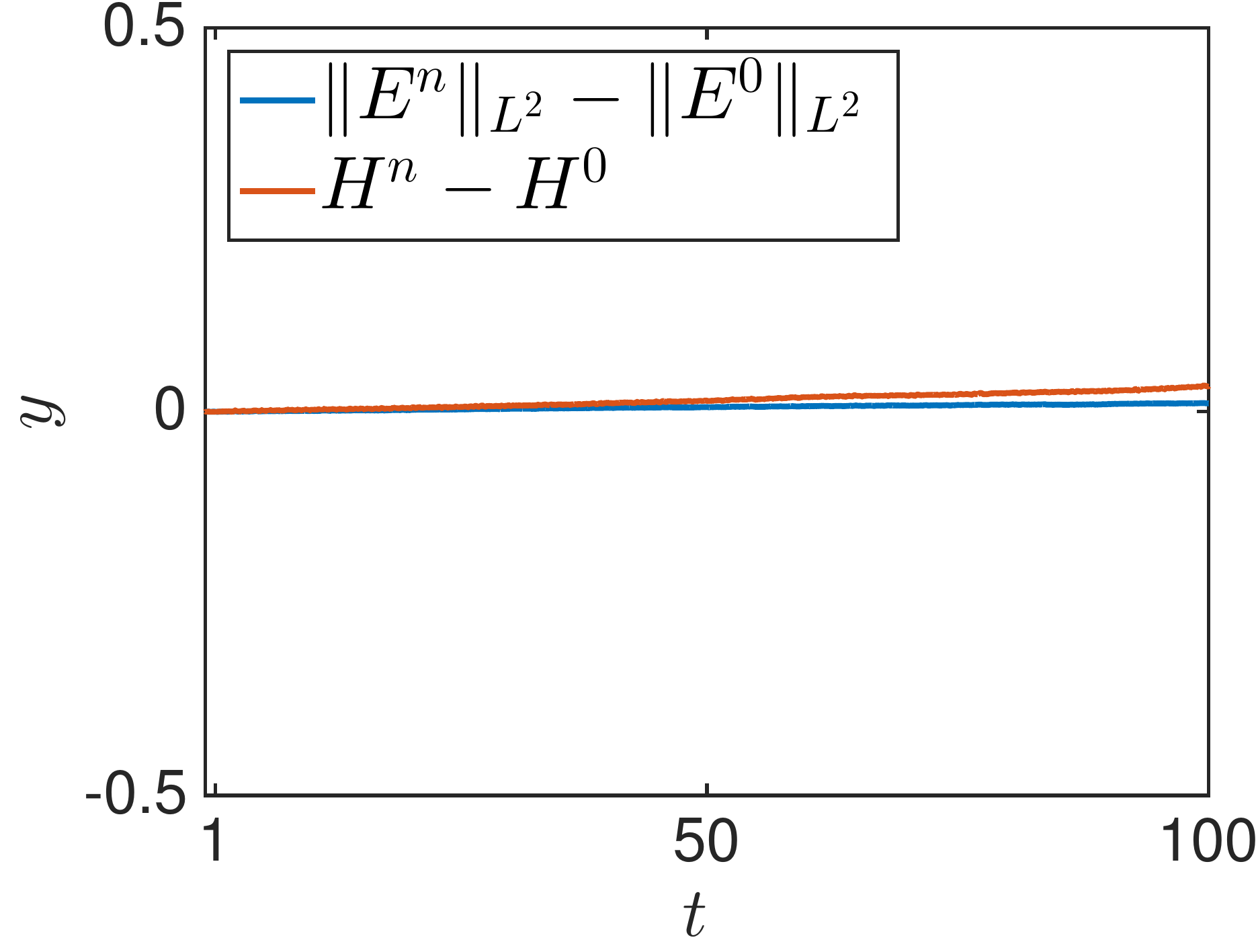}
\hfill
\includegraphics[width=0.43\linewidth]{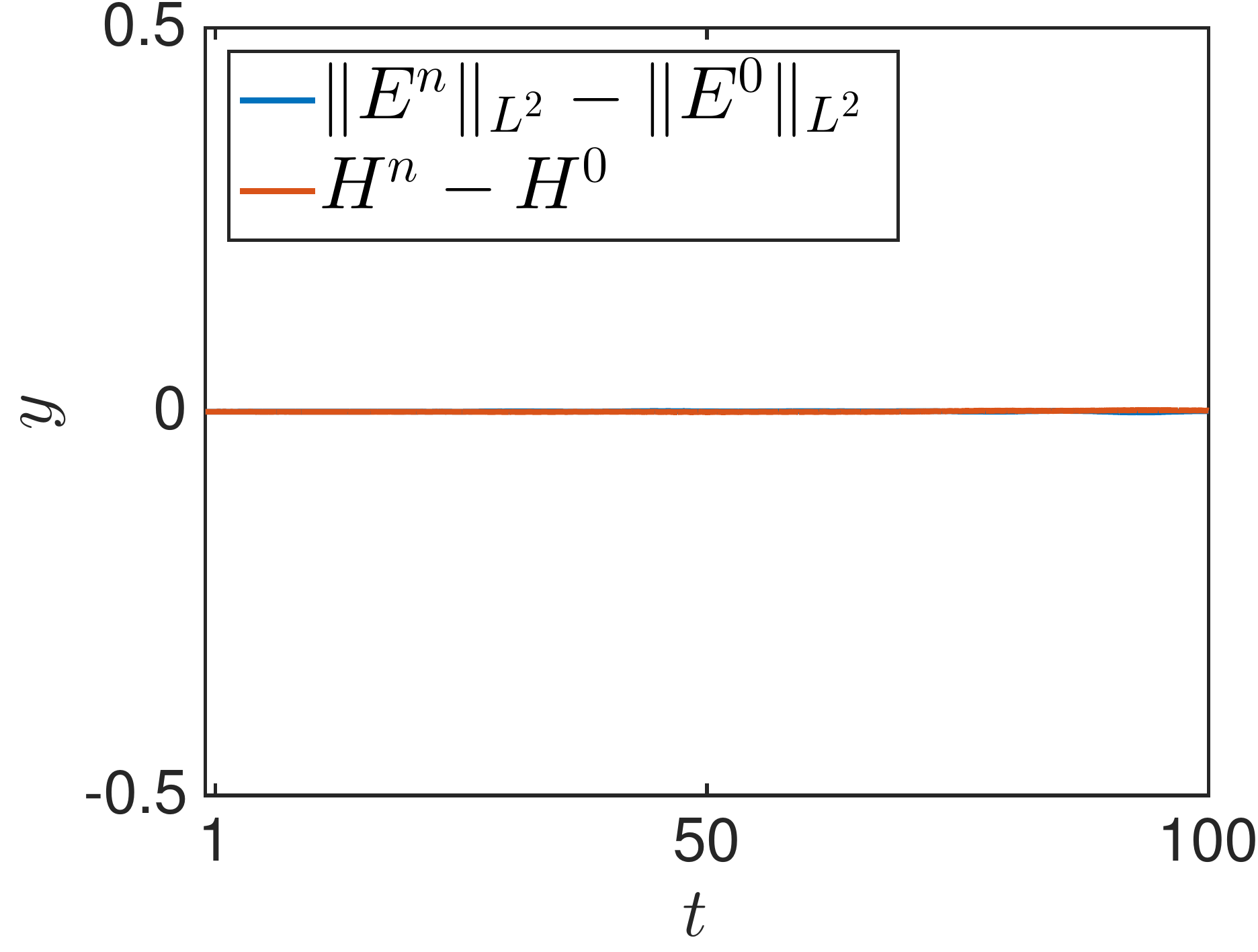}
\caption{Simulation of the deviation of the numerical energy $ H(E^n,u^n,{u^\prime}^n)-H(E^0,u^0,{u^\prime}^0) $ and the $L^2$ norm $ \Vert E^n\Vert_{L^2}-~\Vert E^0\Vert_{L^2}$. Left picture: First-order scheme \eqref{eq:numZ}. Right picture: Second-order scheme \eqref{eq:numZ2}.}\label{fig:Energy}
\end{figure}
\end{example}

\section*{Acknowledgement}
The second author gratefully acknowledges financial support by the Deutsche Forschungsgemeinschaft (DFG) through CRC 1173.

\bibliographystyle{plain}
\bibliography{time-int-zak}

%\end{comment}
\end{document}